\documentclass{amsart}

\DeclareMathSymbol{\twoheadrightarrow}  {\mathrel}{AMSa}{"10}
        \def\GG{{\mathcal G}}

\def\H{{\mathbb H}}
\def\Q{{\mathbb Q}}
\def\Z{{\mathbb Z}}
\def\C{{\mathbb C}}
\def\CC{{\mathfrak C}}
                         \def\cc{{\mathfrak c}}
                   \def\sp{{\mathfrak{sp}}}
                   \def\sL{{\mathfrak{sl}}}
\def\RR{{\mathbb R}}
\def\F{{\mathbb F}}
\def\P{{\mathbb P}}

\def\f{{\tilde F}}
                     \def\f0{{\mathfrak f}}

\def\A8{{\mathbf A}_8}
\def\Alt{\mathrm{Alt}}
                       
                       \def\tr{\mathrm{tr}}
                       \def\Ad{\mathrm{Ad}}
\def\RR{{\mathfrak R}}
\def\Perm{\mathrm{Perm}}
\def\Gal{\mathrm{Gal}}

\def\End{\mathrm{End}}
\def\Aut{\mathrm{Aut}}

\def\Hom{\mathrm{Hom}}

\def\I{{\mathrm{Id}}}
                               \def\MT{{\mathrm{MT}}}
                               \def\mt{{\mathrm{mt}}}

\def\ST{{\mathbf S}}
\def\r{{\mathfrak G}}
\def\fchar{\mathrm{char}}
             \def\red{\mathrm{red}}
\def\GL{\mathrm{GL}}
\def\SL{\mathrm{SL}}
\def\Sp{\mathrm{Sp}}
            \def\Gp{\mathrm{Gp}}
               \def\Lie{\mathrm{Lie}}
\def\M{\mathrm{M}}
\def\A{\mathbf{A}}

\def\dim{\mathrm{dim}}
\def\Oc{{\mathcal O}}

       \def\ZZ{{\mathcal Z}}
       \def\PGL{\mathrm{PGL}}
                            \def\PSL{\mathrm{PSL}}

      \def\g{{\mathfrak g}}

            \def\m{{\mathfrak m}}

                         \def\AA{\mathrm{A}}
                         \def\BB{\mathrm{B}}
                         \def\CC{\mathrm{C}}
                         \def\DD{\mathrm{D}}
\newtheorem{thm}{Theorem}[section]
\newtheorem{lem}[thm]{Lemma}
\newtheorem{cor}[thm]{Corollary}
\newtheorem{prop}[thm]{Proposition}
\theoremstyle{definition}
\newtheorem{defn}[thm]{Definition}

\newtheorem{rem}[thm]{Remark}

        \newtheorem{sect}[thm]{}
\title[Families of absolutely simple hyperelliptic jacobians]{Families of absolutely simple hyperelliptic jacobians}

\author[Yuri G.\ Zarhin]{Yuri G.\ Zarhin}
\address{Department of Mathematics, Pennsylvania State University,
University Park, PA 16802, USA}
\email{zarhin\char`\@math.psu.edu}
\begin{document}

\begin{abstract}
We prove that the jacobian of a hyperelliptic curve $y^2=(x-t)h(x)$ has no
nontrivial endomorphisms over an algebraic closure of the ground field $K$ of
characteristic zero if $t \in K$ and the Galois group of the polynomial $h(x)$
over $K$ is an alternating or symmetric group on  $\deg(h)$ letters and
$\deg(h)$ is an even number $>8$. (The case of odd $\deg(h)>3$ follows easily
from previous results of the author.)

\end{abstract}

\maketitle

\section{Statements}
\label{intro}

 As usual, $\Z$, $\Q$ and $\C$ stand for the ring of integers, the
field of rational numbers
 and the field of complex numbers
respectively. If $\ell$ is a prime then we write $\F_{\ell},\Z_{\ell}$ and
$\Q_{\ell}$ for the ${\ell}$-element (finite) field, the ring of ${\ell}$-adic
integers and field of ${\ell}$-adic numbers respectively. If $A$ is a finite
set then we write $\#(A)$ for the number of its elements.

Let $K$ be a field of characteristic different from $2$, let $\bar{K}$ be its
algebraic closure and $\Gal(K)=\Aut(\bar{K}/K)$ its absolute Galois group. Let
$n\ge 5$ be an integer, $f(x)\in K[x]$  a degree $n$ polynomial {\sl without
multiple roots}, $\RR_f \subset \bar{K}$ the $n$-element set of its roots,
$K(\RR_f) \subset \bar{K}$ the splitting field of $f(x)$ and
$\Gal(f)=\Gal(K(\RR_f)/K)$ the Galois group of $f(x)$ over $K$. One may view
$\Gal(f)$ as a certain group of permutations of $\RR_f$. Let $C_f: y^2=f(x)$ be
the corresponding hyperelliptic curve of genus $\lfloor(n-1)/2\rfloor$. Let
$J(C_f)$ be the jacobian of $C_f$; it is a $\lfloor(n-1)/2\rfloor$-dimensional
abelian variety that is defined over $K$. We write $\End(J(C_f))$ for the ring
of all
 $\bar{K}$-endomorphisms of $J(C_f)$. As usual, we write $\End^{0}(J(C_f))$ for the corresponding
 (finite-dimensional semisimple) $\Q$-algebra $\End(J(C_f))\otimes\Q$.

In \cite{ZarhinMRL,ZarhinMMJ,ZarhinBSMF}
the author proved the following statement.

\begin{thm}
\label{MRL1} Suppose that $\Gal(f)$ is either the full symmetric group $\ST_n$
or the alternating group $\A_n$. Assume also that either $\fchar(K)\ne 3$ or $n
\ge 7$. Then
 $\End(J(C_f))=\Z$. In particular,
$J(C_f)$ is an absolutely simple abelian variety.
\end{thm}

The aim of this note is to discuss the structure of $\End(J(C_f))$ when $f(x)$
has a root in $K$ and the remaining degree $(n-1)$ factor of $f(x)$ has
``large" Galois group over $K$.

\begin{rem}
\label {oneroot} Suppose that $t \in K$ is a root of $f(x)$. By the division
algorithm, $f(x)=(x-t)h(x)$ with $t \in K$ and $h(x)$ a  polynomial of degree
$n-1$ with coefficients in $K$. Then $\RR_f$ is the disjoint union of singleton
$\{t\}$ and the $(n-1)$-element set $\RR_h$ of roots of $h(x)$. Clearly,
$K(\RR_h)=K(\RR_f)$ and $\Gal(h)=\Gal(f)$.
\end{rem}

Our first result is the following  statement.

\begin{thm}
\label{even} Suppose that $n=\deg(f)\ge 6$ is even, $f(x)=(x-t)h(x)$ with $t
\in K$ and $h(x) \in K[x]$. Suppose that $\Gal(h)$ is either the full symmetric
group $\ST_{n-1}$ or the alternating group $\A_{n-1}$. Assume also that either
$\fchar(K)\ne 3$ or $n\ge 8$. Then $\End(J(C_f))=\Z$. In particular, $J(C_f)$
is an absolutely simple abelian variety.
\end{thm}

\begin{proof}
We have $n=2g+2$ where $g$ is the genus of $C_f$ and $n-1=2g+1=\deg(h)$.
Let us consider the polynomials
$$h_1(x)=h(x+t), \ h_2(x)=x^{n-1} h_1(1/x) \in K[x].$$
They all have degree $n-1\ge 5$; in addition, $n-1\ge 7$ if $\fchar(K)=3$. We
have
$$\RR_{h_1}=\{\alpha-t\mid \alpha \in \RR_h\}, \
\RR_{h_2}=\{\frac{1}{\alpha-t}\mid \alpha \in \RR_h\}.$$
This implies that
$$K(\RR_{h_2})=K(\RR_{h_1})=K(\RR_{h})$$
and therefore
$$\Gal(h_2)=\Gal(h_1)=\Gal(h).$$
In particular, $\Gal(h_2)=\ST_{n-1}$ or  $\A_{n-1}$.

Now the equation for $C_f$ may be written down as
$$y^2=(x-t) h_1(x-t).$$
Dividing both sides of the latter equation by $(x-t)^{2(g+1)}$, we get
$$[y/(x-t)^{g+1}]^2=(x-t)^{-(n-1)} h_1(x-t)=h_2(1/(x-t)).$$
 Now the standard
substitution
$$x_1=1/(x-t), \ y_1=y/(x-t)^{g+1}$$
establishes a birational $K$-isomorphism between $C_f$ and a hyperelliptic curve
$$C_{h_2}: y_1^2=h_2(x_1).$$
Now the result follows readily from Theorem \ref{MRL1} applied to the polynomial $h_2(x)$.
\end{proof}

The case of odd $n$ is more difficult.

\begin{thm}
\label{odd} Suppose that $n=\deg(f)\ge 9$ is odd and $f(x)=(x-t)h(x)$ with $t
\in K$ and $h(x) \in K[x]$. Suppose that $\Gal(h)$ is either the full symmetric
group $\ST_{n-1}$ or the alternating group $\A_{n-1}$. Then one of the
 following conditions holds.

\begin{itemize}
\item[(i)]
 $\End^{0}(J(C_f))$ is either $\Q$ or a quadratic field.
 In particular, $J(C_f)$ is an absolutely simple abelian variety.
 \item[(ii)]
 $\fchar(K)>0$ and $J(C_f)$ is a supersingular abelian variety.
 \end{itemize}
\end{thm}

When the genus is at least $5$, we may improve the result as follows.

\begin{thm}
\label{oddZ} Suppose that $n=\deg(f)\ge 11$ is odd and $f(x)=(x-t)h(x)$ with $t
\in K$ and $h(x) \in K[x]$. Suppose that $\Gal(h)$ is either the full symmetric
group $\ST_{n-1}$ or the alternating group $\A_{n-1}$. If $\fchar(K)=0$ then
$\End(J(C_f))=\Z$.
\end{thm}

\begin{rem}
If $K$ is finitely generated over $\Q$ and $h(x)\in K[x]$ is an arbitrary
polynomial of positive even degree without multiple roots then for all but
finitely many $t \in K$ the jacobian of the hyperelliptic curve $y^2=(x-t)h(x)$
is absolutely simple \cite[Theorem 9]{EHALL}. (See also \cite{Masser}.) The
authors of \cite{EHALL} use and compare approaches based on arithmetic geometry
and analytic number theory respectively. In a sense, our approach is purely
algebraic.
\end{rem}

The paper is organized as follows.  Section \ref{group} contains auxiliary
results from group theory. In Sections \ref{AV} and \ref{tate} we study the
structure of  endomorphism algebras
 of abelian varieties with certain Galois properties of points of order $2$.
 Section
 \ref{hyper2} contains an explicit description of the Galois
 module of their points of order $2$ on $J(C_f))$. Combining this description with results
 of Sections \ref{AV} and \ref{tate}, we prove Theorems \ref{odd} and \ref{oddZ}.
In Section \ref{nonconstant} we discuss families of hyperelliptic curves.
 Section \ref{minuscule} contains auxiliary results about $\ell$-adic Lie groups and their
 Lie algebras. In Section \ref{monodromy} we prove  under the conditions of Theorem
 \ref{oddZ} that if the ground field $K$ is finitely generated over $\Q$ then the image
 of $\Gal(K)$ in the automorphism group of the $\ell$-adic Tate module of
 $J(C_f)$ is almost ``as large as possible", namely, it is an open subgroup in the
 group of symplectic similitudes. We use this openness property in order to
 prove for self-products of  $J(C_f)$ the Tate and Hodge conjectures in Sections
 \ref{tateC} and \ref{hodge} respectively. Section \ref{hodge} also contains
 the proof of the Mumford--Tate conjecture for $J(C_f)$.

 I am grateful to the referee, whose comments helped to improve the exposition.

\section{Minimal covers and representations of alternating groups}
\label{group}

\begin{prop}
\label{projAm} Let $m\ge 8$ be an integer, $\A_m$ the corresponding alternating
group. Let $N$ be the smallest positive integer $d$ such that there exists a
group  embedding $\A_m\hookrightarrow \PGL(d,\C)$. Then $N=m-1$.
\end{prop}

\begin{proof}
First, (for all $m$) there exists a well-known group embedding $\A_m
\hookrightarrow \GL(m-1,\C)$, which induces $\A_m\hookrightarrow \PGL(m-1,\C)$.
 Let us consider the non-split short  exact sequence of finite groups
$$1 \to \Z/2\Z \hookrightarrow \A_m^{\prime} \twoheadrightarrow \A_m \to 1$$
where $\A_m^{\prime}$ is the universal central extension of $\A_m$ and $\Z/2\Z$
is the center of $\A_m^{\prime}$. (Recall that $m \ge 8$.) Let $c$ be the only
nontrivial element of the center of $\A_m^{\prime}$.

Now suppose that we are given a group  embedding $\A_m\hookrightarrow
\PGL(d,\C)$. We need to prove that $d \ge m-1$. The universality property of
$\A_m^{\prime}$ implies that the embedding is the projectivization of a
(nontrivial) linear representation
$$\rho^{\prime}:\A_m^{\prime}\hookrightarrow \GL(V)$$
where $V=\C^d$. Replacing $V$ by its $\A_m^{\prime}$-invariant subspace of
minimal dimension with nontrivial action of $\A_m^{\prime}$, we may and will
assume that $\rho^{\prime}$ is a nontrivial irreducible (but not necessary
faithful) representation of $\A_m^{\prime}$. We need to prove that
$\dim_{\C}(V)\ge m-1$.  Schur's Lemma implies that
$$\rho^{\prime}(c)\in \{1,-1\}\subset \C^{*}.$$
 If $\rho^{\prime}(c)=1$ then $\rho^{\prime}$ factors through $\A_m $
and we get a nontrivial linear representation $\A_m\hookrightarrow \GL(V)$,
which must be faithful in light of the simplicity of $\A_m$.  If this is the
case then  $\dim_{\C}(V) \ge m-1$ \cite[p. 71, Theorem 2.5.15]{Encyc}. So,
further, we may and will assume that
$$\rho^{\prime}(c)=-1,$$
and therefore  $\rho^{\prime}$ is faithful, i.e., $\rho$ is a {\sl proper}
projective representation of $\A_m$ \cite[p. 584]{Wagner}. Now an old result of
Schur \cite[S. 250]{Schur} (see also \cite[Th. 1.3(ii)]{Wagner} and \cite[Th. A
on p. 1972]{Klesh}) asserts that $\dim_{\C}(V)\ge m-1$.
\end{proof}

\begin{sect}
\label{cover}
 Recall \cite{FT} that a surjective homomorphism of finite groups
$\pi:\GG_1\twoheadrightarrow \GG$ is called a {\sl minimal cover}
if no proper subgroup of $\GG_1$ maps onto $\GG$ .
In particular, if $\GG$ is perfect and $\GG_1\twoheadrightarrow \GG$ is a
minimal cover then $\GG_1$ is also perfect.
In addition, if $r$ is a positive integer such that every subgroup in $\GG$ of index dividing $r$
coincides with $\GG$ then the same is true for $\GG_1$ \cite[Remark 3.4]{ZarhinMZ}. Namely,
every subgroup in $\GG_1$ of index dividing $r$ coincides with $\GG$.
\end{sect}

\begin{lem}
\label{An} Let $m\ge 5$ be an integer, $\A_m$ the corresponding alternating
group and $\GG_1 \twoheadrightarrow \A_m$
 a minimal cover.

 Then:

 \begin{itemize}
\item[(i)] The only subgroup of index $<m$ in $\GG_1$ is $\GG_1$ itself.
\item[(ii)] Suppose that $m \ge 8$. If $d$ is a positive integer such that
there exists a group embedding $\GG_1\hookrightarrow \PGL(d,\C)$ then $d\ge
m-1$.
\end{itemize}
\end{lem}

\begin{proof}
Let $H$ be a subgroup in $\A_m$ of index $r>1$. Then $\A_m$ acts transitively
on the $r$-element set of (left) $H$-cosets. Therefore there is a nontrivial
homomorphism $\A_m \to \ST_r$, which must be an embedding in light of the
simplicity of $\A_m$. Comparing the orders, we conclude that
$$r! \ge \frac{m!}{2}>(m-1)!$$ and therefore $r \ge m$. This implies that the only subgroup of index $<m$ in
$\A_m$ is $\A_m$ itself. Now arguments of Sect. \ref{cover} imply that
the only subgroup of index $<m$ in $\GG_1$ is $\GG_1$ itself. This proves (i).

Now assume that $m \ge 8$.  By Proposition \ref{projAm}, if $d$ is a positive
integer such that there exists a group embedding $\A_m\hookrightarrow
\PGL(d,\C)$ then $d \ge m-1$.
 Applying Theorem on p. 1092 and Proposition 4.1 (combined with Sect. 4.2) of \cite{FT}, we conclude
 that if  $d$ is a positive integer such that there exists a group embedding
$\GG_1\hookrightarrow \PGL(d,\C)$ then $d\ge m-1$.
\end{proof}

\begin{rem}
\label{Am9} If $m\ge 10$ is an even integer then it follows from results of
Wagner \cite{Wagner2} that every projective representation of $\A_m$ in
characteristic $2$ has dimension $\ge m-2$ \cite[Remark 4.2]{ZarhinMMJ}.

Let $\GG_2\twoheadrightarrow \A_m$ be a surjective homomorphism of finite
groups. Suppose that $\F$ is a field of characteristic $2$ and $d$ a positive
integer such that there exists an embedding $$\GG_2 \hookrightarrow
\PGL(d,\F).$$ I claim that $d \ge m-2$. Indeed, replacing $\GG_2$ by its
suitable subgroup, we may assume that $\GG_2\twoheadrightarrow \A_m$ is a
minimal cover. (E.g., one may take as $\GG_2$ a subgroup of the smallest
possible order that maps surjectively on $\A_m$.) Then the result follows from
a theorem of Feit--Tits \cite[p. 1092]{FT} (see also \cite[Theorem 1]{KL}).
\end{rem}

\begin{cor}
\label{noA}
 Suppose that $m\ge 10$ is an even integer, $L$ is a field that is a finite algebraic
extension of $\Q_2$ and $V$ is a non-zero finite-dimensional vector space over
$L$ with $d:=\dim_L(V)<m-2$. Let $G \subset \Aut_{L}(V)$ be a compact subgroup.

Then there does not exist a surjective continuous homomorphism $G \to \A_m$.
\end{cor}

\begin{proof}
Suppose that there exists a surjective continuous homomorphism $\pi:G
\twoheadrightarrow \A_m$. We write $H$ for $\ker(\pi)$: it is an open normal
subgroup of finite index in $G$ and $G/H \cong \A_m$. The surjectivity of $\pi$
implies that the image of every normal subgroup of $G$ is normal in $\A_m$ and
therefore is either $\{1\}$ or the whole $\A_m$.

 Let $\Oc$ be the ring of integers in $L$. We write $\m$ for
the maximal ideal of $\Oc$ and $\F$ for the (finite) residue field $\Oc/\m$.
Notice that there exists a $G$-stable $\Oc$-lattice $T$ in $V$ of rank $d$.
(Our proof of this assertion follows \cite[Sect. 1.1]{SerreAbelian}.) Indeed,
let $T^{\prime}$ be any $\Oc$-lattice in $V$ of rank $d$ and let $G^{\prime}$
be the set of $s \in G$ such that $s(T^{\prime})=T^{\prime}$. This is an open
subgroup of $G$, because $T^{\prime}$ is an open finitely generated
$\Z_{\ell}$-submodule of $V$. Since $G$ is compact, the compact discrete
$G/G^{\prime}$ is finite. The $\Oc$-lattice $T$ generated by the lattices
$s(T^{\prime}), \ s \in G/G^{\prime}$, is $G$-stable.
 We
have
$$G \subset \Aut_{\Oc}(T)\subset\Aut_{L}(V).$$
We write $G_0$ for the kernel of the reduction map modulo $\m$
$$\red: G \to \Aut_{\Oc}(T/\m T)$$
and $\tilde{G}$ for its image. We have
$$\tilde{G}\subset \Aut_{\Oc}(T/\m T)\cong \GL(d,\F).$$
Clearly, $G_0$ is a pro-$2$-group and $\pi(G_0)$ is a normal $2$-subgroup in
$\A_m$. Since $\A_m$ is simple non-abelian, $\pi(G_0)=\{1\}$. This implies that
$\pi$ factors through a surjective homomorphism
$$\pi_0:\tilde{G}=G/G_0 \twoheadrightarrow \A_m.$$
The surjectivity of $\pi_0$ implies that the center of $\tilde{G}$ goes to the
center of $\A_m$, i.e., $\pi_0$ kills the center of $\tilde{G}$; in particular
$\pi_0$ kills the subgroup $\tilde{Z}$ of scalar matrices in $\tilde{G}$. This
gives us the surjection $\tilde{G}/\tilde{Z} \twoheadrightarrow \A_m$ and the
embedding $\tilde{G}/\tilde{Z}\hookrightarrow \PGL(d,\F)$. It follows from
Remark \ref{Am9} that $d \ge m-2$, which is not the case and we get a desired
contradiction.
\end{proof}

\begin{sect}
Let $g\ge 3$ be an integer. Then $2g \ge 6$ and $\A_{2g}$ is a simple
non-abelian group.

Let $B$ be an $2g$-element set. We write $\Perm(B)$ for the group of all
permutations of $B$. The choice of ordering on $B$ establishes an isomorphism
between $\Perm(B)$ and the symmetric group $\mathrm{S}_{2g}$. We write
$\Alt(B)$ for the only subgroup of index $2$ in $\Perm(B)$. Clearly, every
isomorphism $\Perm(B)\cong\mathrm{S}_{2g}$ induces an isomorphism between
$\Alt(B)$ and the alternating group $\A_{2g}$. Let us consider the
$2g$-dimensional $\F_2$-vector space $\F_2^B$ of all $\F_2$-valued functions on
$B$ provided with the natural structure of faithful $\Perm(B)$-module. Notice
that the standard symmetric bilinear form
$$\F_2^B \times \F_2^B\to \F_2, \ (\phi,\psi)\mapsto \sum_{b\in
B}\phi(b)\psi(b)$$ is non-degenerate and $\Perm(B)$-invariant.

Since $\Alt(B)\subset \Perm(B)$, one may view $\F_2^B$ as a faithful
$\Alt(B)$-module.
\end{sect}

\begin{lem}
\label{AA8}
\begin{itemize}
\item[(i)] The centralizer $\End_{\Alt(B)}(\F_2^B)$ has $\F_2$-dimension $2$.

\item[(ii)] Every proper non-zero $\Alt(B)$-invariant subspace in $\F_2^B$ has
dimension $1$ or $2g-1$. In particular, $\F_2^B$ does not contain a proper
non-zero $\Alt(B)$-invariant even-dimensional subspace.
\end{itemize}
\end{lem}

\begin{proof}
Since $\Alt(B)$ is doubly transitive, (i) follows from \cite[Lemma
7.1]{Passman}.

Notice that the subspace of $\Alt(B)$-invariants
$$M_0:=(\F_2^B)^{\Alt(B)}=\F_2\cdot 1_B,$$
where $1_B$ is the constant function $1$.

In order to prove (ii), recall that
$$M_{0}\subset M_{1}\subset \F_2^B$$
where $M_{1}$ is the hyperplane of functions with zero sum of values. It is
known \cite{Mortimer} that $M_1/M_0$ is a simple $\Alt(B)$-module; clearly,
$\dim(M_1/M_0)=2g-2$.

First, notice that there are no $\Alt(B)$-invariant two-dimensional
 $\F_2$-vector subspaces in $\F_2^B$. Indeed, let $W_2$ be an $\Alt(B)$-invariant
 $\F_2$-vector subspace in $\F_2^B$ with $\dim_{\F_2}(W_2)=2$. Since $\A_{2g}$ is simple
non-abelian and $\GL_2(\F_2)$ is solvable, every homomorphism
$$\Alt(B) \to \Aut_{\F_2}(W_2)\cong \GL_2(\F_2)$$
is trivial and therefore $W_2$ consists of $\Alt(B)$-invariants; however, the
subspace of $\Alt(B)$-invariants in $\F_2^B$ is just one-dimensional.

Second, if $W$ is a $\Alt(B)$-invariant $(2g-2)$-dimensional subspace of
$\F_2^B$ then its orthogonal complement with respect to the standard form is a
two-dimensional $\Alt(B)$-invariant subspace in $\F_2^B$. This implies that
there are no $\Alt(B)$-invariant $(2g-2)$-dimensional  $\F_2$-vector subspaces
in $\F_2^B$.

Let $W$ be a  $\Alt(B)$-invariant subspace of $\F_2^B$ and assume that
$$2<\dim_{\F_2}(W)<2g-2.$$

This implies that the $\Alt(B)$-invariant subspace $$W_1:=W \bigcap M_1$$ is
{\sl not} $\{0\}$. Since $M_1$ is a hyperplane in $\F_2^B$, either $W=W_1$ or
$\dim_{\F_2}(W)=\dim_{\F_2}(W_1)+1$. If $W_1=M_0$ then
$$\dim_{\F_2}(W) \le \dim_{\F_2}(M_0)+1=2,$$
which could not be the case. If $W_1=M_1$ then
$$\dim_{\F_2}(W) \ge \dim_{\F_2}(M_1)=2g-1,$$
which also could not be the case. This implies that
$$W_1 \ne M_0, \ W_1 \ne M_1.$$

Since $M_0$ is a one-dimensional subspace of $M_1$, either $W_1 \supset M_0$ or
$W_1\bigcap M_0=\{0\}$.

In the former case, $W_1/M_0$ is an $\Alt(B)$-invariant subspace of $M_1/M_0$
and the simplicity of $M_1/M_0$ implies that either $W_1/M_0=\{0\}$, i.e.,
$W_1=M_0$ or $W_1/M_0=M_1/M_0$, i.e., $W_1=M_1$. Since $W_1$ is neither $M_0$
nor $M_1$, we conclude that $W_1\bigcap M_0=\{0\}$. We are  going to arrive to
a contradiction. The natural map $W_1\to M_1/M_0$ is an embedding, whose image
is a non-zero $\Alt(B)$-invariant subspace of $M_1/M_0$; the simplicity of the
$\Alt(B)$-module $M_1/M_0$ implies that the image of $W_1$ coincides with the
whole $M_1/M_0$; in particular,
$$\dim_{\F_2}(W_1)=\dim_{\F_2}(M_1/M_0)=2g-2,$$
and we get the $(2g-2)$-dimensional $\Alt(B)$-invariant subspace, which could
not exist. We get the desired contradiction.
\end{proof}

\begin{thm}
\label{simpleA2g} Let $g \ge 3$ be an integer, $B$ a $2g$-element set, $V$ a
$2g$-dimensional $\Q_2$-vector space, $T$ a $\Z_2$-lattice in $V$ of rank $2g$.
Suppose that
$$G \subset \Aut_{\Z_2}(T) \subset \Aut_{\Q_2}(V)$$
is a compact (in the $2$-adic topology) subgroup. Let
$$\tilde{G}:=\red(G)\subset \Aut_{\F_2}(T/2T)$$
be the image of $G$ with respect to the reduction map modulo $2$
$$\red: \Aut_{\Z_2}(T)\to \Aut_{\F_2}(T/2T).$$
Suppose that there exists a group isomorphism $\tilde{G}\cong \Alt(B)$ such
that the $\Alt(B)$-module $T/2T$ is isomorphic to $\F_2^B$.

Then:

\begin{itemize}
\item[(i)] Every proper non-zero $G$-invariant subspace of $V$ has dimension
either $1$ or $2g-1$. \item[(ii)] Assume that $g \ge 5$. Let $\End_{G}(V)$ be
the centralizer of $G$ in $\End_{\Q_2}(V)$. Suppose that $D$  is a semisimple
commutative $\Q_2$-(sub)algebra of $\End_{G}(V)$ (with the same identity
element) such that the $D$-module $V$ is free. Then $D=\Q_2$, i.e., $D$
consists of scalars.
\end{itemize}
\end{thm}

\begin{proof}
The reduction map modulo $2$
$$\red: \Aut_{\Z_2}(T)\to  \Aut_{\F_2}(T/2T)=\Aut_{\F_2}(\F_2^{B})$$
induces a surjective continuous homomorphism
$$\pi: G \twoheadrightarrow \tilde{G}=\Alt(B).$$
 In order to prove (i), let us assume that there exists a
$G$-invariant proper non-zero subspace $V_1 \subset V$ and put $T_1:=V_1\bigcap
T$. Clearly, $T_1$ is a  $G$-invariant free $\Z_2$-submodule of $T$, the
quotient $T/T_1$ is a torsion-free $\Z_2$-module and the $\Z_2$-rank of $T_1$
coincides with the $\Q_2$-dimension of $V_1$. Now, $T_1/2T_1$ is
$\tilde{G}=\Alt(B)$-invariant subspace in $T/2T=\F_2^{B}$, whose
$\F_2$-dimension coincides with the rank of $T$, i.e., with the
$\Q_2$-dimension of $V_1$. It follows from Lemma \ref{AA8}(ii) that
$\dim_{\F_2}(T_1/2T_1)=1$ or $2g-1$. It follows that $\dim_{\Q_2}(V_1)=1$ or
$2g-1$.

In order to prove (ii), first notice that $2g\ge 10$. Let $h$ be the rank of
the free $D$-module $V$ and $e=\dim_{\Q_2}(D)$. Clearly,
$$2g=\dim_{\Q_2}(V)=eh;$$
in particular, $e\mid 2g$ and $h\mid 2g$. It is also clear that for each $u\in
D$ the $\Q_2$-dimension of $u(V)$ is divisible by $h$.

Second, we claim that $h$ is greater than $1$. Indeed, if $h=1$ then $G \subset
\End_D(V)=D$; in particular, $G$ is commutative, which could not be the case
since $G$ maps surjectively onto noncommutative $\A_{2g}$.

Now, assume that $D$ is a field.  If $e=1$ then $D=\Q_2$ and we are done. So
further we assume that $e>1$. Then $V$ carries the natural structure of a
$D$-vector space and $G \subset \Aut_D(V)$. Clearly,
$$\dim_D(V)=\frac{1}{e}\dim_{\Q_2}(V)=\frac{2g}{e} \le \frac{2g}{2}=g<2g-2.$$
Corollary \ref{noA} applied to $m=2g$ and $L=D$ tells us that it could not be
the case.

Now assume that $D$ is {\sl not} a field, i.e., it splits into a direct sum
$D=D_1\oplus D_2$ of two non-zero commutative semisimple $\Q_2$-algebras. Let
$e_i$ be the identity element of $D_i$ for $i=1,2$. Clearly, both $e_i$'s
viewed as elements of $\End_{\Q_2}(V)$ are idempotents; in addition, $e_1
e_2=e_2 e_1=0$. Then $V=V_1\oplus V_2$ where $V_i=e_i (V)$. Clearly, both
$V_i$'s are $G$-invariant; in addition $\dim_{\Q_2}(V_i)$ is divisible by $h$
for $i=1,2$. Since $h>1$ and
$$\dim_{\Q_2}(V_1)+ \dim_{\Q_2}(V_2)=\dim_{\Q_2}(V)=2g,$$ we conclude
that $\dim_{\Q_2}(V_i)\ne 1, 2g-1$. This contradicts to the already proven
assertion (i).
\end{proof}

\section{Abelian varieties}
\label{AV}

Let $F$ be a field, $\bar{F}$ its algebraic closure  and
$\Gal(F):=\Aut(\bar{F}/F)$ the absolute Galois group of $F$.

\begin{lem}
\label{goursat} Let $F_1/F$ and $F_2/F$ be two finite Galois extensions of
fields. Suppose that $G_1=\Gal(F_1/F)$ is a solvable group and
$G_2=\Gal(F_2/F)$ is simple non-abelian. Then $F_1$ and $F_2$ are linearly
disjoint over $F$. In particular, the composition
$$\Gal(F_1 F_2/F_1)\subset \Gal(F_1 F_2/F)\twoheadrightarrow  \Gal(F_2/F)$$
is a group isomorphism $\Gal(F_1 F_2/F_1)\cong \Gal(F_2/F)=G_2$. Here $F_1 F_2$
is the compositum of $F_1$ and $F_2$.
\end{lem}

\begin{proof}
The groups $G_1$ and $G_2$ have no isomorphic quotient except the trivial one.
It follows from Goursat's Lemma \cite[Definition 4.1 and Remark
4.4(ii)]{ZarhinSh} that every subgroup $G$ of $G_1\times G_2$ that maps
surjectively on both factors $G_1$ and $G_2$ must coincide with $G_1\times
G_2$. In order to finish the proof, one has to apply this observation to
$$G=\Gal(F_2 F_2/F)\subset \Gal(F_1/F)\times \Gal(F_2/F)=G_1\times G_2.$$

\end{proof}

If $X$ is an abelian variety of positive dimension
 over $\bar{F}$ then we write $\End(X)$ for the ring of all its
$\bar{F}$-endomorphisms and $\End^0(X)$ for the corresponding $\Q$-algebra
$\End(X)\otimes\Q$. We write $\End_F(X)$ for the ring of all  $F$-endomorphisms
of $X$ and $\End_F^0(X)$ for the corresponding  $\Q$-algebra
$\End_F(X)\otimes\Q$ and $\CC$ for the center of $\End^0(X)$. Both $\End^0(X)$
and $\End_F^0(X)$  are semisimple finite-dimensional $\Q$-algebras.

 The  group
$\Gal(F)$ of $F$ acts on $\End(X)$ (and therefore on $\End^0(X)$)
by ring (resp. algebra) automorphisms and
$$\End_F(X)=\End(X)^{\Gal(F)}, \ \End_F^0(X)=\End^0(X)^{\Gal(F)},$$
since every endomorphism of $X$ is defined over a finite separable
extension of $F$.

If $n$ is a positive integer that is not divisible by $\fchar(F)$ then we write
$X_n$ for the kernel of multiplication by $n$ in $X(\bar{F})$; the commutative
group $X_n$ is a free $\Z/n\Z$-module of rank $2\dim(X)$ \cite[End of Sect. 6,
p. 64]{MumfordAV}. In particular, if $n=2$ then $X_{2}$ is an $\F_{2}$-vector
space of dimension $2\dim(X)$.

 If $X$ is defined over $F$ then $X_n$ is a Galois
submodule in $X(\bar{F})$ and all points of $X_n$ are defined over a finite
separable extension of $F$. We write $\bar{\rho}_{n,X,F}:\Gal(F)\to
\Aut_{\Z/n\Z}(X_n)$ for the corresponding homomorphism defining the structure
of the Galois module on $X_n$,
$$\tilde{G}_{n,X,F}\subset
\Aut_{\Z/n\Z}(X_{n})$$ for its image $\bar{\rho}_{n,X,F}(\Gal(F))$
and $F(X_n)$ for the field of definition of all points of $X_n$.
Clearly, $F(X_n)$ is a finite Galois extension of $F$ with Galois
group $\Gal(F(X_n)/F)=\tilde{G}_{n,X,F}$. If $n=2$  then we get a
natural faithful linear representation
$$\tilde{G}_{2,X,F}\subset \Aut_{\F_{2}}(X_{2})$$
of $\tilde{G}_{2,X,F}$ in the $\F_{2}$-vector space $X_{2}$.

If $F_1/F$ is a finite algebraic extension then $F_1(X_n)$ coincides with the
compositum $F_1 F(X_n)$ of $F_1$ and $F(X_n)$

\begin{lem}
\label{goursat2}
 Let $F_1/F$ be a finite solvable Galois extension of fields.
If $\tilde{G}_{n,X,F}$ is a simple nonabelian group then $F_1$ and $F(X_n)$ are
linearly disjoint over $F$ and $\tilde{G}_{n,X,F_1}=\tilde{G}_{n,X,F}$.
\end{lem}

\begin{proof}
The result follows from Lemma \ref{goursat} combined with the equality
$F_1(X_n)=F_1 F(X_n)$.
\end{proof}

Now and until the end of this Section we assume  that $\fchar(F)\ne 2$. It is
known \cite{Silverberg} that all endomorphisms of $X$ are defined over
$F(X_4)$; this gives rise to the natural homomorphism
$$\kappa_{X,4}:\tilde{G}_{4,X,F} \to \Aut(\End^0(X))$$ and
$\End_F^0(X)$ coincides with the subalgebra
$\End^0(X)^{\tilde{G}_{4,X,F}}$ of $\tilde{G}_{4,X,F}$-invariants
\cite[Sect. 1]{ZarhinLuminy}.

The field inclusion $F(X_2)\subset F(X_4)$ induces a natural
surjection \cite[Sect. 1]{ZarhinLuminy}
$$\tau_{2,X}:\tilde{G}_{4,X,F}\twoheadrightarrow\tilde{G}_{2,X,F}.$$

\begin{defn} We say that $F$ is 2-{\sl balanced} with respect to
$X$ if $\tau_{2,X}$ is a minimal cover. (See \cite{ElkinZ}.)
\end{defn}

\begin{rem}
\label{overL}
 Clearly, there always exists a subgroup $H
\subset \tilde{G}_{4,X,F}$ such that the induced homomorphism
$H\to\tilde{G}_{2,X,F}$ is surjective and a minimal cover. Let us put
$L=F(X_4)^H$. Clearly,
$$F \subset L \subset F(X_4), \ L\bigcap F(X_2)=F$$
and $L$ is a maximal overfield of $F$ that enjoys these properties. It is also
clear that $H$ and $L$ can be chosen that
$$F \subset L \subset F(X_4), \ L\bigcap F(X_2)=F,$$
$$F(X_2)\subset L(X_2),\ L(X_4)=F(X_4), \
\tilde{G}_{2,X,L}=\tilde{G}_{2,X,F}$$
 and $L$ is $2$-{\sl balanced}   with respect to
$X$ (see \cite[Remark 2.3]{ElkinZ}).
\end{rem}

We will need the following three results from previous work.

\begin{thm}
\label{rep}
 Suppose that $E:=\End_F^0(X)$ is a field that contains the
center $\CC$ of $\End^0(X)$. Let $\CC_{X,F}$ be the centralizer of
$\End_F^0(X)$ in $\End^0(X)$.

Then:

\begin{itemize}
\item[(i)] $\CC_{X,F}$ is a central simple $E$-subalgebra in
$\End^0(X)$. In addition, the centralizer of $\CC_{X,F}$ in
$\End^0(X)$ coincides with $E=\End_F^0(X)$ and
$$\dim_E(\CC_{X,F})=\frac{\dim_{\CC}(\End^0(X))}{[E:\CC]^2}.$$
 \item[(ii)] Assume that $F$ is $2$-balanced with
respect to $X$ and $\tilde{G}_{2,X,F}$ is a non-abelian simple group. If
$\End^0(X)\ne E$ (i.e., not all endomorphisms of $X$ are defined over $F$) then
there exist a finite perfect group $\Pi \subset \CC_{X,F}^{*}$ and a surjective
homomorphism $\Pi \to \tilde{G}_{2,X,F}$ that is a minimal cover.
\end{itemize}
\end{thm}

\begin{proof}
This is Theorem 2.4 of \cite{ElkinZ}.
\end{proof}

\begin{lem}
\label{Ksimple} Assume that $X_2$ does not contain  proper non-trivial
$\tilde{G}_{2,X,F}$-invariant even-dimensional subspaces and the centralizer
$\End_{\tilde{G}_{2,X,F}}(X_2)$ has $\F_2$-dimension $2$.

Then $X$ is $F$-simple and $\End_F^0(X)$ is either $\Q$ or a
quadratic field.
\end{lem}

\begin{proof} This is Lemma 3.4  of \cite{ZarhinMA}.
\end{proof}

\begin{lem}
\label{centerDim} Let us  assume that $g:=\dim(X)>0$ and the center of
$\End^0(X)$ is a field, i.e, $\End^0(X)$ is a simple $\Q$-algebra.

Then:

\begin{itemize}
 \item[(i)]
$\dim_{\Q}(\End^0(X))$ divides $(2g)^2$. \item[(ii)] If
$\dim_{\Q}(\End^0(X))=(2g)^2$ then $\fchar(F)>0$ and $X$ is a supersingular
abelian variety.
\end{itemize}
\end{lem}

\begin{proof} This is Lemma 3.5
\footnote{The $Y$ in \cite[Lemma 3.5]{ZarhinMA} should be $X$ and the
$\End^0(Y)$ should be $\End^0(X)$.} of \cite{ZarhinMA}.
\end{proof}

\begin{thm}
\label{mainAV} Let $g\ge 4$ be an integer and $B$ a $2g$-element set. Let $X$
be a $g$-dimensional abelian variety over $F$. Suppose that there exists a
group isomorphism $\tilde{G}_{2,X,F}\cong \Alt(B)$ such that the
$\Alt(B)$-module $X_2$ is isomorphic to $\F_2^B$.

Then one of the following two conditions holds:

\begin{itemize}
\item[(i)] $\End^0(X)$ is either $\Q$ or a quadratic field. In particular, $X$
is absolutely simple. In addition, every finite subgroup of $\Aut(X)$ is
cyclic.

\item[(ii)] $\fchar(F)>0$ and $X$ is a supersingular abelian variety.
\end{itemize}
\end{thm}

 \begin{proof}[Proof of Theorem \ref{mainAV}]
 By Remark \ref{overL}, we may and will assume that $F$ is 2-{\sl balanced} with respect to
 $X$, i.e., $\tau_{2,X}:\tilde{G}_{4,X,F}\twoheadrightarrow
\tilde{G}_{2,X,F}=\A_{2g}$ is a minimal cover. In particular,
$\tilde{G}_{4,X,F}$ is perfect, since $\A_{2g}$ is perfect. Since $\A_{2g}$
does not contain a proper subgroup of index $<2g$,
 it
follows from Lemma \ref{An}(i) that $\tilde{G}_{4,X,F}$  does not contain a
proper subgroup of index $<2g$.
Now Lemmas
\ref{Ksimple} and \ref{AA8} imply that
$\End_F^0(X)$ is either $\Q$ or a quadratic field.

Recall that $\CC$ is the center of $\End^0(X)$.

\begin{lem}
\label{centerK}
 Either $\CC=\Q\subset \End_F^0(X)$ or
$\CC=\End_F^0(X)$ is a quadratic field.
\end{lem}

\begin{proof}[Proof of Lemma \ref{centerK}]
Suppose that $\CC$ is not a field. Then it is a direct sum
$$\CC=\oplus_{i=1}^r \CC_i$$
of number fields $\CC_1, \dots , \CC_r$ with $1<r\le \dim(X)=g$. Clearly, the
center $\CC$ is a $\tilde{G}_{4,X,F}$-invariant subalgebra of $\End^0(X)$; it
is also clear that $\tilde{G}_{4,X,F}$ permutes the summands $\CC_i$'s. Since
$\tilde{G}_{4,X,F}$ does not contain proper subgroups of index $\le g$, each
$\CC_i$ is $\tilde{G}_{4,X,F}$-invariant. This implies that the $r$-dimensional
$\Q$-subalgebra
$$\oplus_{i=1}^r \Q \subset \oplus_{i=1}^r \CC_i$$
consists of $\tilde{G}_{4,X,F}$-invariants and therefore lies in
 $\End_F^0(X)$. It follows that $\End_F^0(X)$ has zero-divisors,
 which is not the case. The obtained contradiction proves that $\CC$
 is a field.

 It is known \cite[Sect. 21]{MumfordAV} that $\CC$ contains a totally real number (sub)field $\CC_0$ with
 $[\CC_0:\Q]\mid \dim(X)$ and such that either $\CC=\CC_0$ or  $\CC$ is a purely imaginary
 quadratic extension of $C_0$. Since $\dim(X)=g$, the degree
 $[\CC_0:\Q]$ divides $g$; in particular, the order of
 $\Aut(\CC_0)$ does not exceed $g$. Clearly,   $\CC_0$ is
 $\tilde{G}_{4,X,F}$-invariant; this gives us the natural
 homomorphism $\tilde{G}_{4,X,F}\to \Aut(\CC_0)$, which must be
 trivial, because its kernel is a (normal) subgroup of index $\le g$ and therefore
  coincides with the whole $\tilde{G}_{4,X,F}$.
Therefore
every element of $\CC_0$ is
 $\tilde{G}_{4,X,F}$-invariant. This implies that
 $\tilde{G}_{4,X,F}$ acts on $\CC$ through a certain group homomorphism
 $\tilde{G}_{4,X,F}\to\Aut(\CC/\CC_0)$  and this homomorphism is
 trivial, because the order of $\Aut(\CC/\CC_0)$ is either $1$ (if
 $\CC=\CC_0$) or $2$ (if $\CC\ne \CC_0$). So, every element of $\CC$
 is $\tilde{G}_{4,X,F}$-invariant, i.e.,
  $$\CC\subset \End^0(X)^{\tilde{G}_{4,X,F}}=\End_F^0(X).$$
  This implies that if $\CC\ne\Q$ then $\End_F^0(X)$ is also not $\Q$
  and therefore is a quadratic field containing $\CC$, which
  implies that $\CC=\End_F^0(X)$ is also a quadratic field.
\end{proof}

It follows that $\End^0(X)$ is a simple $\Q$-algebra (and a
central simple $\CC$-algebra). Let us put $E:=\End_F^0(X)$ and
denote by $\CC_{X,F}$ the centralizer of $E$ in $\End^0(X)$. We have
$$\CC\subset E\subset \CC_{X,F}\subset \End^0(X).$$
Combining Lemma \ref{centerK} with Theorem \ref{rep} and Lemma
\ref{centerDim}, we obtain the following assertion.

\begin{prop}
\label{rep1}
\begin{itemize}
\item[(i)] $\CC_{X,F}$ is a central simple $E$-subalgebra in
$\End^0(X)$,
$$\dim_E(\CC_{X,F})=\frac{\dim_{\CC}(\End^0(X))}{[E:\CC]^2}$$
and $\dim_E(\CC_{X,F})$ divides $(2\dim(X))^2=(2g)^2$.
 \item[(ii)] If $\End^0(X)\ne E$ (i.e., not all endomorphisms of $X$ are
defined over $F$) then there exist a finite perfect group $\Pi
\subset \CC_{X,F}^{*}$ and a surjective homomorphism $\pi:\Pi \to
\tilde{G}_{2,X,F}$ that is a minimal cover.
\end{itemize}
\end{prop}

{\bf End of  Proof of Theorem \ref{mainAV}}.
If $\End^0(X)=E$ then we are done. If $\dim_E(\CC_{X,F})=(2g)^2$ then
$$\dim_{\Q}(\End^0(X))\ge \dim_{\CC}(\End^0(X))\ge
\dim_E(\CC_{X,F})=(2g)^2=(2\dim(X))^2$$ and it follows from Lemma
\ref{centerDim} that $\dim_{\Q}(\End^0(X))=(2\dim(X))^2$, $\fchar(F)>0$ and
$X$ is a supersingular abelian variety. So, further we may and will assume that
$$\End^0(X)\ne E, \ \dim_E(\CC_{X,F}) \ne (2g)^2.$$
We need to arrive to a contradiction.
 Let $\Pi \subset
\CC_{X,F}^{*}$ be as in \ref{rep1}(ii). Since $\Pi$ is perfect,
$\dim_E(\CC_{X,F})>1$. It follows from Proposition
 \ref{rep1}(i) that $\dim_E(\CC_{X,F})=d^2$ where $d$ is a positive integer such that
 $$1<d<2g, \quad d\mid 2g.$$
  This implies that
 $$d \le \frac{2g}{2}=g<2g-2.$$
 Let us fix an embedding $E\hookrightarrow \C$ and an isomorphism
 $\CC_{X,F}\otimes_E\C\cong \M_d(\C)$. This gives us an embedding
 $\Pi\hookrightarrow \GL(d,\C)$. Further we will identify $\Pi$
 with its image in $\GL(d,\C)$. Clearly, only central elements of
 $\Pi$ are scalars. It follows that there is a central subgroup
 $\ZZ$ of $\Pi$ such that the natural homomorphism $\Pi/\ZZ\to
 \PGL(d,\C)$ is an embedding. The simplicity of
$\tilde{G}_{2,X,F}=\A_{2g}$ implies that $Z$ lies in the kernel of $\Pi
\twoheadrightarrow \tilde{G}_{2,X,F}=\A_{2g}$ and the induced map $\Pi/Z \to
\tilde{G}_{2,X,F}$ is also a minimal cover. It follows from Lemma \ref{An}(ii)
applied to $\GG_1=\Pi/\ZZ$ that $d \ge 2g-1$. However, we have seen that
$d<2g-2$. This gives us a desired contradiction.

Recall that $\Aut(X)\subset \End^0(X)^{*}$. Since every finite multiplicative
subgroup in a field is commutative, every finite subgroup of $\Aut(X)$ is
cyclic if $\End^0(X)$ is a field.
\end{proof}

     \section{Tate modules of abelian varieties}
     \label{tate}
     We keep the notation and assumptions of the previous Section.  Let $\ell$ be a prime different from
     $\fchar(F)$. Let us consider the ${\ell}$-adic Tate module
     $T_{\ell}(X)$ of $X$  that is the projective limit of $X_{{\ell}^i}$ ($i=1,2,...$) where
     the transition map $X_{{\ell}^{i+1}}\to X_{{\ell}^i}$ is multiplication by ${\ell}$
      \cite[Sect. 18]{MumfordAV}. The
     Tate module $T_{\ell}(X)$ carries the natural structure of a free $\Z_{\ell}$-module
     of rank $2\dim(X)$. The Galois actions on $X_{{\ell}^i}$ glue together to the
     continuous homomorphism
     $$\rho_{{\ell},X}: \Gal(K) \to \Aut_{\Z_{{\ell}}}(T_{\ell}(X)),$$
     providing $T_{\ell}(X)$ with the structure of Galois module. The natural surjective map
     $T_{\ell}(X)\twoheadrightarrow X_{\ell}$ induces an isomorphism of Galois modules
     $$T_{\ell}(X)\otimes_{\Z_{\ell}} \Z_{\ell}/{\ell \Z_{\ell}}=T_{\ell}(X)/{\ell}T_{\ell}(X)\cong X_{\ell}.$$
     We also consider the $2\dim(X)$-dimensional $\Q_{\ell}$-vector space
     $$V_{\ell}(X):=T_{\ell}(X)\otimes_{\Z_{{\ell}}}\Q_{\ell}.$$
     One may view $T_{\ell}(X)$ as a $\Z_{\ell}$-lattice of rank $2\dim(X)$ in $V_{\ell}(X)$.
     Let us put
     $$G_{{\ell},X}=G_{{\ell},X,F}:=\rho_{{\ell},X}(\Gal(F))\subset
     \Aut_{\Z_{{\ell}}}(T_{\ell}(X))\subset \Aut_{\Q_{\ell}}(V_{\ell}(X));$$
     it is known \cite[Sect. 1.2 and Example 1.2.3]{SerreAbelian} that $G_{{\ell},X}$ is a compact $\ell$-adic Lie
     group, whose Lie algebra $\g_{\ell,X}$ is a $\Q_{\ell}$-Lie subalgebra of $\End_{\Q_{\ell}}(V_{\ell}(X))$.

      The reduction map modulo ${\ell}$
     $$\red: \Aut_{\Z_{{\ell}}}(T_{\ell}(X))\to
     \Aut_{\F_{{\ell}}}(T_{\ell}(X)/{\ell}T_{\ell}(X))=\Aut_{\F_{{\ell}}}(X_{\ell})$$
     induces a continuous surjective homomorphism
     $$\pi_{{\ell},X}:G_{{\ell},X} \twoheadrightarrow \tilde{G}_{{\ell},X,F}\subset \Aut_{\F_{{\ell}}}(X_{\ell}).$$

     \begin{rem}
     \label{isogimage}
     Let $X$ and $Y$ be abelian varieties over $F$ and $u:X \to Y$ be an
     $F$-isogeny. Then $u$ induces, by functoriality, an isomorphism of Galois
     modules $V_{\ell}(X)\cong V_{\ell}(Y)$. This implies that there is a
     continuous group isomorphism between the compact profinite groups
     $G_{{\ell},X,F}$ and $G_{{\ell},Y,F}$.
     \end{rem}

Our next statement deals with the case of $\ell=2$.

\begin{thm}
\label{mainEND} Let $g\ge 5$ be an integer and $B$ a $2g$-element set. Let $F$
be a field of characteristic zero and $X$  a $g$-dimensional abelian variety
over $F$. Suppose that there exists a group isomorphism $\tilde{G}_{2,X,F}\cong
\Alt(B)$ such that the $\Alt(B)=\tilde{G}_{2,X,F}$-module $X_2$ is isomorphic
to $\F_2^B$.

Then  $\End(X)=\Z$.
\end{thm}

\begin{proof}
 By Theorem \ref{mainAV}, we know that $\End^0(X)$ is either $\Q$ or a
quadratic field. If $\End^0(X)=\Q$ then $\End(X)=\Z$. So, further we assume
that $E:=\End^0(X)$ is a quadratic field. Clearly, $\Aut(E)$ is a cyclic group
of order $2$. Replacing if necessary, $F$ by its suitable quadratic extension
and using Lemma \ref{goursat2}, we may and will assume that $\Gal(F)$ acts
trivially on $\End^0(X)$, i.e., all endomorphisms of $X$ are defined over $F$,
i.e., $E=\End^0_{F}(X)$. Clearly, $E_2:=E\otimes_{\Q}\Q_2$ is a two-dimensional
commutative semisimple $\Q_2$-algebra. It is well-known that there is a natural
embedding
$$E_2 \hookrightarrow \End_{\Gal(F)}V_2(X) \subset \End_{\Q_2}(V_2(X)).$$
This implies that $E_2$ sits in the centralizer of $G_{2,X,F}$. It is also
known that the $E_2$-module $V_2(X)$ is free  \cite[Theorem 2.1.1]{Ribet}.
However, applying Theorem \ref{simpleA2g} to $V=V_2(X), T=T_2(X), G=G_{2,X,F}$
 and $D=E_2$, we conclude that $E_2=\Q_2$, which could not be the
case, since $\Q_2$ is the one-dimensional $\Q_2$-algebra.
\end{proof}

\begin{thm}
  \label{nonisog}
  Suppose that $X$ is as in Theorem \ref{mainAV}.   Suppose that $Y$ is a $g$-dimensional abelian variety over $F$ that enjoys
  one of the following properties:

   \begin{itemize}
   \item[(i)]
   $\tilde{G}_{2,Y,F}$ is solvable;
   \item[(ii)]
   The fields $F(X_2)$ and $F(Y_2)$ are linearly disjoint over $F$.
   \end{itemize}

   If $\fchar(F)=0$ then $X$ and $Y$ are not isogenous over $\bar{F}$.
  \end{thm}

  \begin{proof}[Proof of Theorem \ref{nonisog}]
  If $\tilde{G}_{2,Y,F}$ is solvable then it follows from Lemma \ref{goursat}
that $F(X_2)$ and $F(Y_2)$ are linearly disjoint over $F$, since
$\tilde{G}_{2,X,F}=\Alt(B)\cong\A_{2g}$ is simple non-abelian. So, $F(X_2)$ and
$F(Y_2)$ are linearly disjoint over $F$. Since the compositum of $F(X_2)$ and
$F(Y_2)$ coincides with $F(Y_2)(X_2)$, we conclude that
$$\tilde{G}_{2,X,F}=\tilde{G}_{2,X,F(Y_2)}.$$
  Replacing the ground field $F$ by $F(Y_2)$, we may and will assume that  $\tilde{G}_{2,Y,F}=\{1\}$, i.e.,
   the Galois module $Y_2$ is trivial. It follows that
   $G_{2,Y,F} \subset \I +2 \End_{\Z_2}(T_2(Y))$; in particular, $G_{2,Y,F}$ is
   a pro-$2$-group.
By Theorem  \ref{mainAV}, $\End^0(X)$ is either $\Q$ or a quadratic field say,
$L$. In the former case all the endomorphisms of $X$ are defined over $F$. In
the latter case, all the endomorphisms of $X$ are defined either over $F$ or
over a certain quadratic extension of $F$, because the automorphism group of
$L$ is the cyclic group of order $2$. Replacing if necessary $F$ by the
corresponding quadratic extension, we may and will assume
    that all the endomorphisms of $X$ are defined over $F$.
    In particular, all the automorphisms of $X$ are defined over
    $F$.

   There is still a continuous surjective
   homomorphism $$G_{2,X,F}\twoheadrightarrow \tilde{G}_{2,X,F}=\Alt(B) \cong \A_{2g}$$
   and therefore $G_{2,X,F}$ is {\sl not}  a pro-$2$-group. This implies that
   there  does {\sl not} exist a continuous isomorphism between $G_{2,X,F}$
   and $G_{2,Y,F}$. It follows from Remark \ref{isogimage} that $X$ and $Y$ are
   {\sl not} isogenous over $F$.

     Let $u:X \to Y$ be an $\bar{F}$-isogeny of abelian varieties.  As we have seen, $u$ could {\sl not} be defined over $F$.
     However, there exists a finite Galois extension $F_u/F$ such that $u$
    is defined over $F_u$.

    Let us consider the $1$-cocycle
    $$c:\Gal(F_u/F)\to \End^0(X)^{*}, \ \sigma \mapsto c_{\sigma}:= u^{-1} \sigma(u).$$
    Since the Galois group acts trivially on $\End^0(X)$, the map $c: \Gal(F_u/F) \to \End^0(X)^{*}$ is a group homomorphism,
    whose image is a finite subgroup of $\End^0(X)^{*}$ and therefore is cyclic, thanks to Theorem  \ref{mainAV}(i).
    Therefore there is a finite cyclic subextension $F^{\prime}/F$ such that
    $$c_{\sigma}=1 \ \forall \sigma \in \Gal({F_u}/F^{\prime})\subset \Gal(F_u/F).$$
    It follows that $u$ is defined over $F^{\prime}$ and therefore the  $\Gal(F^{\prime})$-modules $V_2(X)$ and $V_2(Y)$ are isomorphic.
    In particular, there is a continuous group isomorphism between $G_{2,X,F^{\prime}}$
   and $G_{2,Y,F^{\prime}}$. But $G_{2,X,F^{\prime}}\subset G_{2,X,F}$ is still a
   pro-$2$-group. On the other hand, since $F^{\prime}/F$ is cyclic and
   $\Gal(F(X_2)/F)=\tilde{G}_{2,X,F}\cong \A_{2g}$ is simple non-abelian, Lemma
   \ref{goursat2} tells us that
   $$\tilde{G}_{2,X,F^{\prime}}=\tilde{G}_{2,X,F}\cong \A_{2g}$$
   and there is a surjective continuous homomorphism
   $$G_{2,X,F^{\prime}}\twoheadrightarrow \tilde{G}_{2,X,F^{\prime}}\cong \A_{2g}.$$
   In particular, $G_{2,X,F^{\prime}}$ is {\sl not}  a pro-$2$-group.
   there does not exist a continuous group isomorphism
     \end{proof}

\section{Points of order 2}
\label{hyper2}
\begin{sect}
\label{heart} Let $K$ be a field of characteristic different from $2$, let
$f(x)\in K[x]$ be a polynomial of {\sl odd} degree $n\ge 5$ and without
multiple roots. Let $C_f$ be the hyperelliptic curve $y^2=f(x)$ and $J(C_f)$
the jacobian of $C_f$.
  The Galois module $J(C_f)_2$ of points of order $2$ admits the following
  description.

  Let $\F_2^{\RR_f}$ be the $n$-dimensional $\F_2$-vector space of functions $\varphi: \RR_f \to \F_2$
   provided with the
   natural structure of $\Gal(f)\subset \Perm(\RR_f)$-module. The canonical surjection
    $$\Gal(K)\twoheadrightarrow \Gal(K(\RR_f)/K)=\Gal(f)$$
    provides $\F_2^{\RR_f}$ with the structure of $\Gal(K)$-module. Let us
    consider the hyperplane
    $$(\F_2^{\RR_f})^0:=\{\varphi:\RR_f \to \F_2\mid \sum_{\alpha\in \RR_f}\varphi(\alpha)=0\}\subset \F_2^{\RR_f}.$$
    Clearly, $(\F_2^{\RR_f})^0$ is a Galois submodule in $\F_2^{\RR_f}$.

    It is well-known (see, for instance, \cite{ZarhinTexel}) that if $n$ is odd
    then the Galois modules $J(C_f)_2$ and
$(\F_2^{\RR_f})^0$ are isomorphic. It follows that if $X=J(C_f)$ then
$\tilde{G}_{2,X,K}=\Gal(f)$ and $K(J(C_f)_2)=K(\RR_f)$.
\end{sect}

\begin{lem}
\label{order2} Suppose that $n=\deg(f)$ is odd and $f(x)=(x-t)h(x)$ with $t\in
K$ and $h(x)\in K[x]$. Then $\tilde{G}_{2,J(C_f),K}\cong \Gal(h)$ and the
Galois modules $J(C_f)_2$ and $\F_2^{\RR_h}$ are isomorphic.
\end{lem}

\begin{proof}
We have
$$\tilde{G}_{2,X,K}=\Gal(f)=\Gal(h).$$
 In order to prove the second assertion, it suffices to check that the Galois modules $(\F_2^{\RR_f})^0$ and
$\F_2^{\RR_h}$ are isomorphic. Recall that $\RR_{f}$ is the disjoint union of
$\RR_{h}$ and $\{t\}$. Consider the map
 $(\F_2^{\RR_f})^0\to  \F_2^{\RR_h}$ that sends the function $\varphi: \RR_f \to \F_2$
 to its restriction to $\RR_h$. Obviously, this map is an isomorphism of Galois modules.
\end{proof}

\begin{cor}
\label{helpcor} Suppose that $\fchar(K)=0$, $n=\deg(f)=2g+1$ is odd  and
$f(x)=(x-t)h(x)$ with $t\in K$ and $h(x)\in K[x]$. Assume also that
$\Gal(h)=\Alt(\RR_h)\cong \A_{2g}$.

\begin{itemize}
\item[(i)] If $g \ge 4$ then $\End^0(J(C_f))$ is either $\Q$ or a quadratic
field.
 \item[(i)]
If $g \ge 5$ then $\End(J(C_f))=\Z$.
\end{itemize}
\end{cor}

\begin{proof}[Proof of Corollary \ref{helpcor}]
Let us put $K=F$, $X=J(C_f)$ and $B=\RR_h$.
 Then assertion (i) is an immediate corollary of Lemma \ref{order2}
and Theorem \ref{mainAV}. The assertion (ii) follows from Theorem
\ref{mainEND}.

\end{proof}

\begin{proof}[Proof of Theorem \ref{odd} and Theorem \ref{oddZ}]
Replacing if necessary, $K$ by its suitable quadratic extension, we may and
will assume that $\Gal(h)=\A_{2g}$ (recall that $n=2g+1$). Now  Theorem
\ref{odd} is an immediate corollary of Corollary \ref{helpcor}(i).
Theorem \ref{oddZ} follows from Corollary \ref{helpcor}(ii).
\end{proof}

\begin{thm}
 Suppose that $\fchar(K)=0$, $n=\deg(f)\ge 9$ is odd and $f(x)=(x-t)h(x)$ with $t \in K$
and $h(x) \in K[x]$.  Assume also that $\Gal(h)$ is either $\ST_{n-1}$ or
$\A_{n-1}$. Suppose that $f_1(x) \in K[x]$ is a degree $n$ polynomial without
multiple roots  that enjoys one of the following properties:

\begin{itemize}
\item[(i)]
$f_1(x)$ splits into a product of linear factors over $K$.

\item[(ii)] $f_1(x)=(x-t_1)h_1(x)$ with $t_1 \in K$ and $h_1(x) \in K[x]$. In
addition, the splitting fields of $h(x)$ and $h_1(x)$ are linearly disjoint
over $K$.

\item[(iii)] The splitting fields of $h(x)$ and $f_1(x)$ are linearly disjoint
over $K$.
\end{itemize}

Then the jacobians $J(C_f)$ and $J(C_{f_1})$ are not isogenous over $\bar{K}$.
\end{thm}

\begin{proof}
It suffices to do the case when the splitting fields of $f(x)$ and $f_1(x)$ are
linearly disjoint over $K$. (This condition is obviously fulfilled in the cases
(i) and (ii).) Let us put $X=J(C_f), \ Y=J(C_{f_1})$. According to Sect.
\ref{heart},
$$K(J(C_f)_2)=K(\RR_f), \ K(J(C_{f_1})_2)=K(\RR_{f_1}).$$
Now the result follows from Theorem \ref{nonisog} combined with Lemma
\ref{order2}.
\end{proof}

\begin{thm}
Suppose that $\fchar(K)=0$, $n=2g+2=\deg(f)\ge 10$ is even and
$f(x)=(x-t_1)(x-t_2)u(x)$ with
$$t_1, t_2 \in K, \ t_1 \ne t_2, \ u(x)\in K[x], \ \deg(u)=n-2.$$
If $\Gal(u)=\ST_{n-2}$ or $\A_{n-2}$ then $\End^0(J(C_f))$ is either $\Q$ or a
 quadratic field; in particular, $J(C_f)$ is an absolutely simple
abelian variety. In addition, if $n\ge 12$ then $\End(J(C_f))=\Z$.
\end{thm}

\begin{proof}
Let us put $h(x)=(x-t_2)u(x)$. We have $f(x)=(x-t_1)h(x)$. As in the proof of
Theorem \ref{even}, let us consider the degree $(n-1)$ polynomials

$$h_1(x)=h(x+t_1)=(x+t_1-t_2)u(x+t_1), \ h_2(x)=x^{n-1} h_1(1/x) \in K[x].$$
 We have
$$\RR_{h_1}=\{\alpha-t_1\mid \alpha \in \RR_{h}\}=
\{\alpha-t_1+t_2\mid \alpha \in \RR_{u}\} \bigcup \left\{t_2-t_1\right\},$$
$$\RR_{h_2}=\left\{\frac{1}{\alpha-t_1}\mid \alpha \in \RR_u\right\}\bigcup
\left\{\frac{1}{t_2-t_1}\right\}.$$ This implies that
$$K(\RR_{h_2})=K(\RR_{h_1})=K(\RR_{u})$$
and
$$h_2(x)=\left(x-\frac{1}{t_2-t_1}\right)v(x)$$
where $v(x)\in K[x]$ is a degree $(n-2)$ polynomial with
$K(\RR_{v})=K(\RR_{u})$; in particular, $\Gal(v)=\Gal(u)=\ST_{n-2}$ or
$\A_{n-2}$. Again,  the standard substitution
$$x_1=1/(x-t_1), \ y_1=y/(x-t_1)^{g+1}$$
establishes a birational $K$-isomorphism between $C_f$ and a hyperelliptic
curve
$$C_{h_2}: y_1^2=h_2(x_1).$$
Now the result follows from Theorems \ref{odd} and \ref{oddZ} applied to
$h_2(x_1)$.
\end{proof}

\section{Families of hyperelliptic curves}
\label{nonconstant}

 Throughout this Section,  $K$ is a field of characteristic
different from $2$,
 $\bar{K}$  its algebraic closure and $\Gal(K)=\Aut(\bar{K}/K)$ its absolute
Galois group. Let $n\ge 5$ be an integer, $f(x)\in K[x]$  a degree $n$
polynomial {\sl without multiple roots}, $\RR_f \subset \bar{K}$ the
$n$-element set of its roots, $K(\RR_f) \subset \bar{K}$ the splitting field of
$f(x)$ and $\Gal(f)=\Gal(K(\RR_f)/K)$ the Galois group of $f(x)$ over $K$. One
may view $\Gal(f)$ as a certain group of permutations of $\RR_f$. Let $C_f:
y^2=f(x)$ the corresponding hyperelliptic curve of genus
$\lfloor(n-1)/2\rfloor$.

\begin{thm}
\label{curveshyper} Let $n\ge 7$ be an integer,  $h(x)\in K[x]$ an irreducible
polynomial of degree $n-1$, whose Galois group is either $\ST_{n-1}$ or
$\A_{n-1}$. For every $t\in K$ let $f_t(x)=(x-t)h(x)$ and $D(t)=C_{f_t}$. Then
there exists a finite set $B=B(h)$ that enjoys the following properties.

If $t_1$ and $t_2$ are distinct elements of $K$ such that the hyperelliptic
curves $D(t_1)$ and $D(t_2)$ are isomorphic over $\bar{K}$ then both $t_1$ and
$t_2$ belong to $B$.
\end{thm}

\begin{rem}
It is well known (\cite[Ch. 2, Sect. 3, pp. 253--255]{GH}, \cite[Ch. VIII,
Sect. 3]{Dolgachev}) that the hyperelliptic curves $D(t_1)$ and $D(t_2)$ are
isomorphic over $\bar{K}$ if and only if there exists a fractional linear
transformation $T\in \PGL_2(\bar{K})=\Aut(\P^1)$ that sends the branch points
of the canonical double cover $D(t_1)\to \P^1$ to the branch points of the
canonical double cover $D(t_2)\to \P^1$.
\end{rem}

\begin{proof}[Proof of Theorem \ref{curveshyper}]
Let $\RR_h\subset \bar{K}$ be the $(n-1)$-element set of roots of $h(x)$. Let
us consider two distinct ordered triples
$$\{\alpha_1,\alpha_2,\alpha_3\}, \{\beta_1,\beta_2,\beta_3\}\subset \RR_h\subset
\bar{K}\subset\P^1(\bar{K})$$ of roots of $h(x)$. There exists exactly one
fractional-linear transformation
$$T=T(\alpha_1,\alpha_2,\alpha_3;\beta_1,\beta_2,\beta_3)\in  \PGL_2(\bar{K})$$
such that
$$T(\alpha_1)=\beta_1,  T(\alpha_2)=\beta_2,
T(\alpha_3)=\beta_3.$$  We write $J_1=J_1(h)$ for the set of all those $T$'s
(for all choices of $\alpha$'s and $\beta$'s). Clearly, $J_1$ is a finite
subset in $\PGL_2(\bar{K})$, whose cardinality bounded by a constant depending
only on $n$. It is also clear that $J_1$ does {\sl not} contain the identity
element and
$$J_1=\{T^{-1}\mid T \in J_1\}.$$
In addition, $J_1$ is $\Gal(K)$-stable.

\begin{lem}
\label{curveshyperL} Suppose that $\Gal(h)$ is either $\ST_{n-1}$ or
$\A_{n-1}$. Then the only fractional-linear transformation $U\in
\PGL_2(\bar{K})$ that sends $\RR_h$ into itself is the identity map.
\end{lem}

\begin{proof}[Proof of Lemma \ref{curveshyperL}]
Replacing if necessary $K$ by its suitable quadratic extension, we may and will
assume that $\Gal(h)=\A_{n-1}$. Let us consider the subgroup
$$G=\{U \in \PGL_2(\bar{K})\mid U(\RR_h)=\RR_h\}\subset \PGL_2(\bar{K}).$$
Since $\#(\RR_h)=n-1\ge 6>3$, $G$ is a finite group and the natural
homomorphism $G \to \Perm(\RR_h)=\ST_{n-1}$ is injective; we write
$$G_0\subset \Perm(\RR_h)=\ST_{n-1}$$
for its image. Clearly, $G_0 \cong G$. Since $\RR_h$ is Galois-invariant, $G_0$
is stable with respect to the conjugation by elements of $\Gal(K)$. Since the
image of $\Gal(K)$ in $\Perm(\RR_h)$ is $\A_{n-1}$, the subgroup $G_0\subset
\ST_{n-1}$ is stable with respect to conjugation by elements of $\A_{n-1}$.
Since $n-1>5$, it follows that $G_0$ is either trivial or
 $\ST_{n-1}$ or $\A_{n-1}$. Therefore  $G$ is either trivial or isomorphic to
$\ST_{n-1}$ or $\A_{n-1}$. Now the classifications of finite subgroups of
$\PSL_2(\bar{K})=\PGL_2(\bar{K})$ tells us that $G$ is trivial and we are done.
\end{proof}

\begin{prop}
\label{curveshyperP} Suppose that $\Gal(h)$ is either $\ST_{n-1}$ or
$\A_{n-1}$. If $T\in J_1$ then there exists $\sigma\in\Gal(K)$ such that
$\sigma(T)\ne T$.
\end{prop}

\begin{proof}[Proof of Proposition \ref{curveshyperP}]
Assume that $\sigma(T)= T$ for all $\sigma\in\Gal(K)$.
 Since $T\in J_1$, there exists $\alpha\in \RR_h$ such that
$T(\alpha)\in \RR_h$. Since $\RR_h$ coincides with the Galois orbit of
$\alpha$, we have $T(\RR_h)=\RR_h$. By Lemma \ref{curveshyperL}, $T$ is the
identity map. But $J_1$ does not contain the identity element. The obtained
contradiction proves the desired result.
\end{proof}

{\bf Continuing the proof of Theorem \ref{curveshyper}}.

Let $B_1$ be the set of all $t\in K$ such that there exist $\alpha\in\RR_h$ and
$T\in J_1$ such that $t=T(\alpha)$. Let $B_2$ be the set of all $t\in K$ such
that there exist $\sigma \in\Gal(K)$ and $T\in J_1$ such that $\sigma(T)\ne T$
and  $t$ is a fixed point of $\sigma(T)T^{-1}$. Let $B_3$ be the set of all
$t\in K$ such that there exists $T\in J_1$ such that $t=T(\infty)$.

Clearly,  $B_1$, $B_2$ and $B_3$ are finite sets, whose cardinalities are
bounded by a constant depending only on $n$. Let us put
$$B=B(h)=B_1 \cup B_2 \cup B_3.$$

{\bf End of the proof of Theorem \ref{curveshyper}}. First assume that $n$ is
even. Then the set of the ramification points of the double cover
$$D(t) \to \P^1, \ (x,y) \mapsto x$$
coincides with $\RR_{f_t}=\{t\}\cup \RR_h$. So, if $t_1 \ne t_2$ and $D(t_1)$
and $D(t_2)$ are isomorphic over $\bar{K}$ then there exists $T \in
\PGL_2(\bar{K})$ such that
$$T(\{t_1\}\cup \RR_h)=\{t_2\}\cup \RR_h.$$
By Lemma \ref{curveshyperL}, $T(\RR_h)\ne  \RR_h$. This implies that there is
$\beta \in \RR_h$ with $T(\beta)=t_2$. It follows that $T(\RR_h\setminus
\{\beta\}) \subset \RR_h$. Since $\#(\RR_h\setminus \{\beta\})=n-2>3$, the
transformation $T$ lies in $J_1$ and therefore $t_2\in B_1$. By symmetry, $t_1$
also lies in $B_1$.

Now assume that $n$ is odd. Then the set of the ramification points of the
double cover
$$D(t) \to \P^1, \ (x,y) \mapsto x$$
coincides with $\{\infty\}\cup \RR_{f_t}=\{\infty\}\cup \{t\}\cup \RR_h$. So,
if $t_1 \ne t_2$ and $D(t_1)$ and $D(t_2)$ are isomorphic over $\bar{K}$ then
there exists $T \in \PGL_2(\bar{K})$ such that
$$T(\{\infty\}\cup\{t_1\}\cup \RR_h)=\{\infty\}\cup\{t_2\}\cup \RR_h.$$
By Lemma \ref{curveshyperL}, $T(\RR_h)\ne  \RR_h$. So, either there is $\beta
\in \RR_h$ such that  $T(\beta)=t_2$ or $T(\RR_h)$ does {\sl not} contain $t_2$
but contains $\infty$. In the former case, arguments as above prove that
$t_2\in B_1$. In the latter case, either $T(t_1)=t_2$ or $T(\infty)=t_2$. In
the former case, $\sigma(T)(t_1)=t_2$ for all $\sigma \in \Gal(K)$ and
therefore $t_2$ is a fixed point
 of $\sigma(T)T^{-1}$, which implies that $t_2
\in B_2$. In the latter case, $T(\infty)=t_2$ and therefore $t_2 \in B_3$.
However, we always have
$$t_2 \in B_1 \cup B_2 \cup B_3=B.$$
By symmetry, $t_1\in B$.

\end{proof}

\begin{thm}
Suppose that $n \ge 8$ is an integer that does not equal  $9$. Let $K$ be a
field of characteristic zero and $h(x) \in K[x]$  an irreducible polynomial of
degree $n-1$. Assume also that $\Gal(h)$ is either $\ST_{n-1}$ or $\A_{n-1}$.
Then there exists a finite set $B=B(h)\subset K$ that enjoys the following
properties.

Let $t_1$ and $t_2$ be distinct elements of $K$ and let $f_1(x)=(x-t_1)h(x), \
f_2(x)=(x-t_2)h(x)$. If the abelian varieties $J(C_{f_1})$ and $J(C_{f_2})$ are
isomorphic over $\bar{K}$ then both $t_1$ and $t_2$ belong to $B$.
\end{thm}

\begin{proof}
Let $B(h)$ be as in Theorem \ref{curveshyper}. By Theorems \ref{even} and
\ref{oddZ}, $\End(J(C_{f_1}))=\Z$ and $\End(J(C_{f_2})))=\Z$. This implies that
both jacobians $J(C_{f_1})$ and $J(C_{f_2})$ have exactly one principal
polarization and therefore a $\bar{K}$-isomorphism  of abelian varieties
$J(C_{f_1})\cong J(C_{f_2})$ respects the principal polarizations. Now the
Torelli theorem implies that the hyperelliptic curves $C_{f_1}$ and $C_{f_2}$
 are isomorphic over $\bar{K}$.
  It follows from Theorem \ref{curveshyper}  that both $t_1$ and $t_2$ belong to $B(h)$.

\end{proof}

\section{Compact groups and simple Lie algebras}
\label{minuscule}

\begin{lem}
\label{kerimage}
 Let $G$ be a compact group, $g\ge 3$  an integer
  and $\pi:G \twoheadrightarrow \A_{2g}$ a continuous
surjective group homomorphism. Let $M$ be a finite group and $\pi^{\prime}: G
\to M$ be a continuous group homomorphism. Let us put $H:=\ker(\pi^{\prime})$.
Suppose that one of the following conditions holds:
\begin{itemize}
\item[(i)] $M$ is solvable. \item[(ii)] The order of $M$ is strictly less than
the order of $\A_{2g}$.

\item[(iii)] The homomorphism $\pi^{\prime}$ is surjective and there does not
exist a surjective group homomorphism $M \twoheadrightarrow \A_{2g}$.
\end{itemize}

Then $H$ is a normal open  compact subgroup of finite index in $G$ and
$\pi(H)=\A_{2g}$, i.e., $\pi:H\to \A_{2g}$ is a surjective continuous
homomorphism.
\end{lem}

\begin{proof}
The first assertion about $H$ is obvious. Replacing $M$ by its subgroup
 $\pi^{\prime}(G)$ we may and will assume that $\pi^{\prime}$ is surjective,
 i.e., $M=\pi^{\prime}(G)$. In particular,   $G/H \cong M$.

 In order to prove the existence of the homomorphism, it suffices to do the case (iii),
  because in both cases (i) and (ii) a surjection $M\twoheadrightarrow\A_{2g}$ does not exist, since  $\A_{2g}$
  is simple non-abelian.
So, let us assume that there are no  surjective group homomorphisms from $G$ to
$M$. The normality of $H$ and surjectiveness of $\pi$ imply that $\pi(H)$ is
normal in $\A_{2g}$, i.e., either $\pi(H)=\A_{2g}$ or $\pi(H)=\{1\}$. If
$\pi(H)=\A_{2g}$ then we are done. If $\pi(H)=\{1\}$ then $\pi$ factors through
$G/H=M$ and we get a surjective homomorphism $M \twoheadrightarrow \A_{2g}$.
\end{proof}

\begin{sect}
Let $V$ be a non-zero even-dimensional $\Q_{\ell}$-vector space,
$$e:V \times V \to \Q_{\ell}$$
an alternating nondegenerate $\Q_{\ell}$-bilinear form,
$$\Sp(V,e)=\Aut(V,e)=\{s \in \Aut_{\Q_{\ell}}(V)\mid e(sx,sy)=e(x,y) \
\forall x,y\in V\} \subset \Aut_{\Q_{\ell}}(V)$$ the corresponding symplectic
group, viewed as a closed $\ell$-adic Lie subgroup in
 $\Aut_{\Q_{\ell}}(V)$ and
$$\sp(V,e):=\Lie(\Sp(V,e)) =$$
$$ \{u \in \End_{\Q_{\ell}}(V)\mid e(ux,y)+e(x,uy)=0 \ \forall x,y\in V\}\subset
\End_{\Q_{\ell}}(V)$$ its Lie algebra, viewed as a $\Q_{\ell}$-Lie subalgebra
in $\End_{\Q_{\ell}}(V)$. It is well-known that
$$\Sp(V,e)\subset \SL(V), \ \sp(V,e)\subset \Lie(\SL(V))=\sL(V):=\{u\in \End_{\Q_{\ell}}(V)\mid
\tr_V(u)=0\}$$ where $\tr_V:\End_{\Q_{\ell}}(V) \to \Q_{\ell}$ is the trace
map. Let us consider the group of symplectic similitudes
$$\Gp(V,e)=\{s \in \Aut_{\Q_{\ell}}(V)\mid \exists c \in \Q_{\ell}^{*} \ \mbox{such that } e(sx,sy)=
c\cdot  e(x,y)$$ $$ \ \forall x,y\in V\}\subset \Aut_{\Q_{\ell}}(V);$$ it is
also an $\ell$-adic Lie (sub)group, whose Lie algebra coincides with
$\Q_{\ell}\I\oplus \sp(V,e)$ where $\I:V\to V$ is the identity map. The map $s
\mapsto c$ defines the $\ell$-adic Lie group homomorphism
$$\chi_e:\Gp(V,e)\to \Q_{\ell}^{*},$$
whose kernel coincides with $\Sp(V,e)$.

Let $G$ be a compact subgroup in $\Aut_{\Q_{\ell}}(V)$. The compactness implies
that $G$ is closed. By the $\ell$-adic version of Cartan's theorem \cite[Part
II, Ch. V, Sect. 9, p. 155]{SerreLie}, $G$ is an $\ell$-adic Lie (sub)group. We
write $\Lie(G)$ for the Lie algebra of $G$. Then $\Lie(G)$ is a $\Q_{\ell}$-Lie
subalgebra of $\End_{\Q_{\ell}}(V)$. If $G \subset \Sp(V,e)$ or $G \subset
\Gp(V,e)$ then $\Lie(G)\subset \sp(V,e)$ or $\Lie(G)\subset \Q_{\ell}\I\oplus
\sp(V,e)$ respectively.
\end{sect}

\begin{lem}
\label{groupal}
 Let $G$ be a compact subgroup in $\Aut_{\Q_{\ell}}(V)$. Suppose
that for every open subgroup $G^{\prime}\subset G$ of finite index the
$G^{\prime}$-module $V$ is absolutely simple. Then there exists a semisimple
$\Q_{\ell}$-Lie algebra $\g^{ss}\subset \End_{\Q_{\ell}}(V)$ that enjoys the
following properties:
\begin{itemize}
\item[(i)] $\Lie(G)=\g^{ss}$ or $ \Q_{\ell}\I\oplus \g^{ss}$.
 \item[(ii)] The
$\g^{ss}$-module $V$ is absolutely simple. \item[(iii)] If $G\subset \Gp(V,e)$
then
$$\g^{ss}=\Lie(G)\bigcap\sp(V,e)\subset \sp(V,e).$$
 In addition, if $G^0=\ker(\chi_e:G \to
\Q_{\ell}^{*})$ then $\Lie(G^{0})=\g^{ss}$.
\end{itemize}
\end{lem}

\begin{proof}Clearly,
$\Lie(G)=\Lie(G^{\prime})$ for all $G^{\prime}$.  Since the $G$-module $V$ is
(semi)simple, it follows from  \cite[Proposition 1]{SerreDiv} that $\Lie(G)$ is
reductive, i.e., $\Lie(G)= \g^{ss}\oplus \cc$ where $\cc$ is the center of
$\Lie(G)$ and $\g^{ss}$ is a semisimple $\Q_{\ell}$-Lie algebra. I claim that
$\End_{\Lie(G)}(V)=\Q_{\ell}\I$. Indeed, let $G^{\prime}$ be an open subgroup
of finite index that is sufficiently small in order to lie in the image of the
exponential map. Then $\End_{G^{\prime}}(V)=\End_{\Lie(G)}(V)$. In light of the
{\sl absolute} simplicity of the $G^{\prime}$-module $V$, we have
$\End_{G^{\prime}}(V)=\Q_{\ell}\I$ and therefore
$$\End_{\Lie(G)}(V)=\End_{G^{\prime}}(V)=\Q_{\ell}\I.$$
Since the center $\cc$ lies in $\End_{\Lie(G)}(V)$, either $\cc=\{0\}$ and
$\Lie(G)=\g^{ss}$ or $\cc=\Q_{\ell}\I$ and $\Lie(G)=\Q_{\ell}\I\oplus\g^{ss}$.
In both cases
$$\Q_{\ell}\I=\End_{\Lie(G)}(V)=\End_{\g^{ss}}(V).$$
Since $\g^{ss}$ is semisimple, the $\g^{ss}$-module $V$ is absolutely simple.
This proves (i) and (ii). In order to prove (iii), notice that in both cases
the semisimple $\g^{ss}=[\Lie(G),\Lie(G)]$. Clearly, $\Lie(G)\subset
\Q_{\ell}\I\oplus\sp(V,e)$. Taking into account that $\sp(V,e)$ is a simple Lie
algebra, we have
$$\g^{ss}=[\Lie(G),\Lie(G)]\subset [\Q_{\ell}\I\oplus\sp(V,e),\Q_{\ell}\I\oplus\sp(V,e)]=\sp(V,e).$$
It follows easily that
$$\g^{ss}=\Lie(G)\bigcap \sp(V,e).$$
 In order to compute $\Lie(G^0)$, notice that according to
\cite[Part II, Ch. V, Sect. 2, p. 131]{SerreLie}, the Lie algebra of the kernel
of $\chi_e$ coincides with the kernel of the corresponding tangent map
$\Lie(G)\to \Lie(\Q_{\ell}^*)=\Q_{\ell}$. It follows that either
$\Lie(G^0)=\Lie(G)$ or $\Lie(G^0)$ is a Lie subalgebra of codimension $1$ in
$\Lie(G)$. On the other hand, since $G^0=\Sp(V,e)\bigcap G$, we have
$$\Lie(G^0)\subset \Lie(G)\bigcap \sp(V,e)=\g^{ss}.$$
So, if $\Lie(G)=\Q_{\ell}\I\oplus\g^{ss}$ then the (co)dimension arguments
imply that $\Lie(G^0)=\g^{ss}$. If $\Lie(G)=\g^{ss}$ then the tangent map is
the zero map, because $\Lie(G)$ is semisimple and $\Q_{\ell}$ is commutative.
This implies that the kernel of the tangent map coincides with the whole
$\Lie(G)$, i.e.,
$$\Lie(G^0)=\Lie(G)=\g^{ss}.$$
\end{proof}

\begin{thm}
\label{simpleLie}
 Let $g \ge 5$ be an integer, $V$  a $2g$-dimensional vector
space over $\Q_{2}$. Let $G \subset \SL(V)\subset \Aut_{\Q_2}(V)$ be a compact
$2$-adic Lie group that enjoys the following properties:

\begin{itemize}
\item[(i)] There exists a continuous surjective homomorphism $\pi:G
\twoheadrightarrow \A_{2g}$;

\item[(ii)] Let $\g \subset \End_{\Q_2}(V)$ be the $\Q_{2}$-Lie algebra of $G$.
Then the $\g$-module $V$ is absolutely simple, i.e., the natural representation
of $\g$ in $V$ is irreducible and $\End_{\g}(V)=\Q_{2}$.
\end{itemize}

Then:
\begin{itemize}
\item[(i)] The $\Q_2$-Lie algebra $\g$ is absolutely simple.

\item[(ii)] Let $L/\Q_2$ be a finite Galois field extension such that the
$L$-Lie algebra $\g_L:=\g\otimes_{\Q_2}L$ is split. (Such $L$ always exists.)
If $W$ is a faithful simple $\g_L$-module of finite $L$-dimension then
$\dim_{L}(W)\ge 2g-2$.
\end{itemize}
\end{thm}

\begin{proof}
Our plan is to apply Corollary \ref{noA} to a certain compact $2$-adic  Lie
group that will be obtained from $G$ in several steps. At every step we replace
the group either by the kernel of a homomorphism to a ``small"  finite group or
by the image of a homomorphism, whose kernel is a pro-$2$-group.

 Since $G \subset \SL(V)$, we have $\g \subset \sL(V)$.
 Clearly, $\g$ is reductive, its center is either $\{0\}$ or the scalars. Since
$\g\subset \sL(V)$, the center of $\g$ is $\{0\}$.  This implies that $\g$ is
{\sl semisimple}. Let $\r\subset \GL(V)$ be the connected semisimple linear
algebraic (sub)group over $\Q_2$, whose Lie algebra coincides with $\g$.

 Let us prove that the semisimple $\Q_2$-Lie algebra $\g$ is, in fact, absolutely simple.
Indeed, there exists a finite Galois field extension $L/\Q_2$ such that the
semisimple $L$-Lie algebra $\g_L=\g\otimes_{\Q_2}L$ is split; in particular,
$\g_L$ splits into a (finite) direct sum
$$\g_L=\oplus_{i\in I} \ g_i$$
of absolutely simple split $L$-Lie algebras $\g_i$. Here $I$ is the set of
minimal ideals $\g_i$ in $\g$. It is well-known that the $L$-vector space
$$V_L=V\otimes_{\Q_2}L$$
becomes an absolutely simple faithful $\g_L$-module and splits into a tensor
product
$$V_L=\otimes_{i\in I}\ W_i$$
of absolutely simple faithful $\g_i$-modules $W_i$. Since each $\g_i$ is simple
and $W_i$ is faithful,
$$\dim_L(W_i)\ge 2 \ \forall i\in I$$
and therefore
$$n:=\#(I)\le \log_2(\dim_L(V_L))=\log_2(\dim_{\Q_2}(V))=\log_2(2g)<g.$$
Let us consider the adjoint representation
$$\Ad: G \to \Aut(\g)\subset \Aut(\g_L).$$
Since the $\g$-module $V$ is absolutely simple, $\ker(\Ad)$ consists of
scalars. Since $G\subset \SL(V)$, the group $\ker(\Ad)$ is finite commutative.
Clearly, $G$ permutes elements of $I$ and therefore gives rise to the
continuous homomorphism (composition)
$$\pi_1:G \stackrel{\Ad}{\to} \Aut(\g)\subset \Aut(\g_L) \to \Perm(I) \cong \ST_n.$$
Let $G_1$ be the kernel of $\pi_1$: it is an open normal subgroup of finite
index in $G$ and therefore the Lie algebra of $G_1$ coincides with $\g$. Since
$n<g$,
$$\#(\ST_n)=n! < \frac{1}{2}(2g)!=\#(\A_{2g}).$$
It follows from Lemma \ref{kerimage} applied to $\pi^{\prime}=\pi_1$ and
$M=\ST_n$,
 that  $\pi(G_1)= \A_{2g}$. So, we may replace $G$ by $G_1$
and assume that $G$ leaves stable every $g_i$. This means that the image of
$$G \stackrel{\Ad}{\to} \Aut(\g_L)=\Aut(\oplus_{i\in I}\g_i)$$
lies in $\prod_{i\in I}\Aut(\g_i)$.

We write $\r_L\subset \GL(V_L)$ for the connected semisimple linear algebraic
(sub)group over $L$ obtained from $\r$ by extension of scalars. Clearly, the
$L$-Lie algebra of $\r_L$ coincides with $\g_L$. We write $\r_i$ for the
simply-connected absolutely simple split $L$-algebraic group, whose Lie algebra
coincides with $\g_i$ \cite{Springer}. We write $\r^{\Ad}\subset \GL(\g_L)$ for
the adjoint group of $\r$ and $\Ad_{\r}$ for the corresponding central isogeny
$\r \to \r^{\Ad}$. If $\r_i^{\Ad}\subset \GL(\g_i)$ is the adjoint group of
$\r_i$ then
$$\r^{\Ad}=\prod_{i\in I} \ \r_i^{\Ad}.$$
Since $g_i$ is simple split,  the group $\r_i^{\Ad}(L)$ is a closed (in the
$2$-adic topology) normal subgroup in $\Aut(\g_i)$ of index $1$, $2$ or $6$
\cite[Ch. 9, Sect. 4, Th. 4 and Remark on p. 281]{Jackobson}. Let us consider
the composition
$$G \stackrel{\Ad}{\to} \prod_{i\in I}\Aut(\g_i) \twoheadrightarrow \prod_{i\in I}\
\Aut(\g_i)/\r_i^{\Ad}(L).$$ Its image is a finite solvable group. Let $G_2$ be
its kernel: it is an open normal subgroup of finite index in $G$ and the Lie
algebra of $G_1$ coincides with $\g$. It follows from Lemma \ref{kerimage}
applied to $\pi^{\prime}=\pi_2$  that  $\pi(G_2) = \A_{2g}$. So, we may replace
$G$ by $G_2$ and assume that the image  of $G \stackrel{\Ad}{\to} \prod_{i\in
I}\ \Aut(\g_i)$ lies in $\prod_{i\in I}\ \r_i^{\Ad}(L)$. So, we get the
continuous group homomorphism
$$\gamma:G \stackrel{\Ad}{\twoheadrightarrow}  \prod_{i\in I}\ \r_i^{\Ad}(L)=\r^{\Ad}(L).$$
Recall that $\ker(\Ad)$ is finite commutative. This implies that $\ker(\gamma)$
is a finite commutative group. On the other hand, let us consider the canonical
central isogeny of semisimple $L$-algebraic groups
$$\alpha:=\Ad_{\r}:\r=\prod_{i\in I} \r_i \to \prod_{i\in I} \r_i^{\Ad}=\r^{\Ad}.$$
Applying \cite[Corollary 3.4(2) on p. 409]{ZarhinMMJ}, we conclude that there
exists a compact subgroup $G_3\subset \r(L)=\prod_{i\in I} \r_i(L)$ and a
surjective continuous homomorphism $\pi_3:G_3 \twoheadrightarrow \A_{2g}$,
whose kernel $H$ is an open normal subgroup of finite index in $G_3$.
 Applying \cite[Prop. 3.3 on pp. 408--409]{ZarhinMMJ}\footnote{The $R_j$ in \cite[p. 409, line 3]{ZarhinMMJ} should be $S_j$.} to
$\ell=2,F=L,G=G_3, S_i=\r_i(L)$, we conclude that there exist $j \in I$ and a
compact subgroup $G_4 \subset \r_i(L)$ provided with surjective continuous
homomorphism $G_4 \twoheadrightarrow \A_{2g}$.

Let $W$ be a finite-dimensional $L$-vector space that carries the structure of
absolutely simple faithful $\g$-module. I claim that
$$\dim_{L}(W) \ge 2g-2.$$
Indeed, there exists a $L$-rational representation
$$\rho_W: \r_i \to \GL(W),$$
whose kernel is a finite central subgroup. The composition
$$\pi_4:G_4 \subset \r_i(L) \stackrel{\rho_W}{\to} \Aut_L(W)$$
is a continuous group homomorphism, whose kernel is a finite central subgroup
of $G_4$. Clearly, the central subgroup $\ker(\pi_4)$ is killed by the
surjective homomorphism $G_4 \twoheadrightarrow \A_{2g}$. So, if we put
$G_5=\pi_4(G_4)\subset \Aut_L(W)$ then $G_5$ is a compact subgroup that admits
a surjective continuous homomorphism $G_5 \twoheadrightarrow \A_{2g}$. It
follows from Corollary \ref{noA} that
$$\dim_L(W) \ge 2g-2.$$
In particular,  $\dim_L(V_j)\ge 2g-2$. On the other hand,
$$2g=\dim_{\Q_2}(V)=\dim_L(V_L)=\prod_{i\in I} \dim_L(V_i)=$$
$$\dim_L(V_j)\prod_{i\in I, i\ne j} \dim_L(V_i)\ge (2g-2)2^{\#(I)-1}.$$ It
follows that $\#(I)=1$, i.e., $I=\{j\}$ and $\g_L=\g_j$. This means that $\g_L$
is an absolutely simple $L$-Lie algebra and therefore $\g$ is an absolutely
simple $\Q_{2}$-Lie algebra.
\end{proof}

\begin{cor}
\label{cor0} Let $g \ge 5$ be an integer, $V$  a $2g$-dimensional vector space
over $\Q_{2}$ and
$$e:V \times V \to \Q_{2}$$
an alternating nondegenerate $\Q_2$-bilinear form. Let $G \subset
\Gp(V,e)\subset\Aut_{\Q_2}(V)$ be a compact $2$-adic Lie group that enjoys the
following properties:

\begin{itemize}
\item[(i)]
 For every open subgroup $G^{\prime}\subset G$ of finite
index the $G^{\prime}$-module $V$ is absolutely simple.

\item[(ii)] There exists a continuous surjective homomorphism $\pi:G
\twoheadrightarrow \A_{2g}$.
\end{itemize}

Let $G^0$ be the kernel of $\chi_e:G \to \Q_{\ell}^*$ and
$$\g:=\Lie(G^0)\subset \End_{\Q_2}(V).$$

Then:
\begin{itemize}
\item[(i)] There exists a continuous surjective homomorphism $\pi:G^{0}
\twoheadrightarrow \A_{2g}$.
        \item[(ii)] The $\Q_2$-Lie algebra $\g:=\Lie(G^0)$
is absolutely simple and the $\g$-module $V$ is absolutely simple. In addition,
$\g=\Lie(G)\bigcap \sp(V,e)$ and the $\g$-module $V$ is symplectic.
                  \item[(iii)] Either $\Lie(G)=\g$ or $\Lie(G)=\Q_2\I\oplus \g$.

          \item[(iv)] Let $L/\Q_2$ be a finite Galois field extension such that
the simple $L$-Lie algebra $\g_L=\g\otimes_{\Q_2}L$ is split. If $W$ is a
faithful simple $\g_L$-module of finite $L$-dimension then $\dim_{L}(W)\ge
2g-2$.
\end{itemize}
\end{cor}

\begin{proof}
In order to prove (i), notice that $\pi(G_0)$ is a normal subgroup in
$\A_{2g}$, i.e., $\pi(G_0)$ is either $\A_{2g}$ or $\{1\}$. In the former case
we are done, so let us assume that $\pi(G_0)=\{1\}$. This means that $\pi$
kills $G_0$ and therefore induces a surjective homomorphism $G/G_0
\twoheadrightarrow \A_{2g}$. On the other hand, $G/G_0$ is isomorphic to the
subgroup $\chi_e(G)$ of $\Q_{\ell}^{*}$ and therefore is commutative. It
follows that $\A_{2g}$ is also commutative, which is not the case. This
contradiction proves (i).

 Lemma \ref{groupal} implies all the assertions in (ii) and (iii) except the
 absolute simplicity of $\g^{ss}=\g$; however, it implies that $\g$ is
 semisimple and that the $\g$-module $V$ is absolutely simple. Now the the
 absolute simplicity of $\g$ follows from Theorem \ref{simpleLie} applied to
 $G^0$. The assertion (iv) also follows from Theorem \ref{simpleLie} applied to
 $G^0$.
\end{proof}

We refer to \cite[Ch. VIII, Sect. 7.3]{Bourbaki} for the notion and basic
properties of {\sl minuscule} weights. (See also \cite{SerreHT} and
\cite{ZarhinW}.)

\begin{cor}
\label{LieSP} We keep all notation and assumption of  Corollary \ref{cor0}.
 Let $L/\Q_2$ be a finite Galois field extension such that $\g_L$ is split.
Assume also that $\g_L$ is a classical simple Lie algebra, the $\g_L$-module
$V_L=V\otimes_{\Q_2}L$ is fundamental and its highest weight is a minuscule
weight.

 Then $\g=\sp(V,e)$. In particular, $\g$ is an absolutely simple $\Q_2$-Lie algebra of of type $\CC_g$.
\end{cor}

\begin{proof}
In the course of the proof we will freely use Tables from \cite{Bourbaki}. It
follows from Corollary \ref{cor0} and Theorem \ref{simpleLie} that if $W$ is a
faithful simple $\g_L$-module of finite $L$-dimension then $\dim_{L}(W)\ge
2g-2$.

Extending the form $e$ by $L$-linearity to $V_L$, we obtain the alternating
nondegenerate $L$-bilinear form
$$e_L:V_L \times V_L \to L.$$
Clearly,
$$\g_L \subset \sp(V_L,e_L);$$
in particular, $V_L$ is the symplectic $\g_L$-module.

Let $r$ be the rank of the absolutely simple classical $L$-Lie algebra $\g_L$.

If $\g_L$ is of type $\AA_r$  then there exists a $(r+1)$-dimensional
$L$-vector
 space $W$ such that $\g_L\cong\sL(W)$ and the fundamental $\g_L$-module $V_L$ is
 isomorphic to $\wedge^i_{L}(W)$ for suitable $i$ with $1\le i \le r$.
 However,
 $$r+1=\dim_{L}(W)\ge 2g-2 \ge 8;$$
 in particular, $r \ge 8$. Since $V_L$ is symplectic,
 $$2 \le i \le r-1.$$
 We have
 $$2g=\dim_{\Q_2}(V)=\dim_L(V_L))=\dim_L(\wedge^i_{L}(W))\ge
 \dim_L(\wedge^2_{L}(W)= \frac{(r+1)r}{2}\ge $$
 $$4(r+1)>(r+1)+4>(2g-2)+3>2g.$$
 The obtained contradiction proves that $\g$ is {\sl not} of type $\AA_r$.

 If $\g_L$ is of type $\BB_r$ then there exists a $(2r+1)$-dimensional $L$-vector
 space $W$ such that $W$ is an (orthogonal) absolutely simple faithful $\g_L$-module and
 every absolutely simple $\g_L$-module
 with minuscule highest weight has dimension  $2^r$. Hence
 $$\dim_L(V_L)=2^{r}.$$
 We have
 $$2g=\dim_L(V_L)=2^{r}, \ 2r+1=\dim_L(W)\ge 2g-2.$$
 This implies that $g=2^{r-1}$ and therefore $r> 3$, since $g\ge 5$. We also have
 $2r+1=\dim_L(W)\ge 2g-2= 2^r-2$, which implies that $2r+1 \ge 2^r-2$. It
 follows that $r \le 3$ and therefore $g=2^{r-1}\le 4$, which could not be the
 case. The obtained contradiction proves that $\g$ is {\sl not} of type $\BB_r$.

 If $\g_L$ is of type $\CC_r$ then the only absolutely simple (symplectic) $\g_L$-module,
 with minuscule highest weight has dimension $2r$. This implies that $\dim_L(V_L)=2r$
 and therefore
 $\dim_L(\g_L)=\dim_L(\sp(V_L,e_L)).$ Since $\g_L\subset\sp(V_L,e_L)$, we
 conclude that
 $$\g_L=\sp(V_L,e_L)\subset \End_L(V_L).$$
 Now the dimension arguments imply that $\g=\sp(V,e)$.

 If $\g_L$ is of type $\DD_r$ (with $r\ge 3$) then there exists a $2r$-dimensional $L$-vector
 space $W$ such that $W$ is an (orthogonal) absolutely simple faithful $\g_L$-module; on
 the other hand,
 every absolutely simple symplectic $\g_L$-module,
 with minuscule highest weight must have dimension $2^{r-1}$. This implies that
 $$\dim_L(V_L)=2^{r-1}.$$
 We have
 $2g=\dim_L(V_L)=2^{r-1}$ and therefore $g=2^{r-2}$. Since $g \ge 5$ we have
 $r\ge 5$. We have
 $$ 2r=\dim_L(W)\ge 2g-2=2^{r-1}-2$$ and therefore $2r \ge 2^{r-1}-2$, which
 could not be the case, since $r\ge 5$. The obtained contradiction proves that $\g$ is {\sl not} of type $\DD_r$.
\end{proof}

\section{Abelian varieties without nontrivial endomorphisms}
\label{monodromy}
 We use the notation of Sections \ref{AV}   and \ref{tate}.
 Let $F$ be a field and $\ell$ a prime different from $\fchar(F)$. As usual, we
write $\Z_{\ell}(1)$ for the projective limit of the cyclic multiplicative
groups
$$\mu_{\ell^i}:=\{z\in \bar{F}^{*}\mid z^{\ell^i}=1\}.$$
The group $\Z_{\ell}(1)$ carries the natural structure of free
$\Z_{\ell}$-module of rank $1$. It also carries the natural structure of the
Galois module provided by the $\ell$-adic {\sl cyclotomic} character
$$\chi_{\ell}:\Gal(K)\to \Z_{\ell}^{*}=\Aut_{\Z_{\ell}}(\Z_{\ell}(1)).$$
Let us fix a (non-canonical) isomorphism of  $\Z_{\ell}$-modules
$$\Z_{\ell}(1)\cong \Z_{\ell}.$$

Let $X$ be an abelian variety of positive dimension over $F$ and let $\lambda$
be a  polarization on $X$ that is defined over $F$. Then $\lambda$ gives rise
to a Riemann form \cite[Sect. 20]{MumfordAV}
$$e_{\lambda}:T_{\ell}(X)\times T_{\ell}(X)\to \Z_{\ell}(1)\cong \Z_{\ell};$$
$e_{\lambda}$ is a nondegenerate $\Gal(K)$-equivariant alternating
$\Z_{\ell}$-bilinear form. Here the equivariance means that
$$e_{\lambda}(\sigma(x),\sigma(y))=\chi_{\ell}(\sigma)\cdot e_{\lambda}(x,y) \
\forall \sigma \in \Gal(K).$$ Extending $e_{\lambda}$ by $\Q_{\ell}$-linearity,
we obtain a nondegenerate $\Gal(K)$-equivariant alternating
$\Q_{\ell}$-bilinear form
$$e_{\lambda}:V_{\ell}(X)\times V_{\ell}(X)\to \Q_{\ell}$$
such that
$$e_{\lambda}(\sigma(x),\sigma(y))=\chi_{\ell}(\sigma)\cdot e_{\lambda}(x,y) \
\forall x,y\in V_{\ell}(X), \sigma \in \Gal(K).$$ This implies that
$$G_{\ell,X}=G_{{\ell},X,F}=\rho_{\ell,X}(\Gal(K)) \subset \Gp(V_{\ell}(X),e_{\lambda}).$$
We write $\g_{\ell,X}\subset \End_{\Q_{\ell}}(V_{\ell}(X))$ for the
$\Q_{\ell}$-Lie algebra of the compact $\ell$-adic Lie group $G_{\ell,X,F}$.
Clearly,
$$\g_{\ell,X}\subset \Lie(\Gp(V_{\ell}(X),e_{\lambda}))=\Q_{\ell}\I\oplus
\sp(V_{\ell}(X),e_{\lambda}).$$

\begin{rem}
\label{selfproduct} Let $m$ be a positive integer and $X^m$ is the $m$th power
(self-product) of $X$. Then
$$V_{\ell}(X^m)=\bigoplus_{j=1}^{m}V_{\ell}(X)=:V_{\ell}(X)^m$$
(the equality of Galois modules) and
$$\g_{\ell,X^m}=\g_{\ell,X} \subset \End_{\Q_{\ell}}(V_{\ell}(X))\subset
\M_m(\End_{\Q_{\ell}}(V_{\ell}(X)))= \End_{\Q_{\ell}}(V_{\ell}(X)^m)$$ (the
last embedding is the diagonal one). The polarization $\lambda$ gives rise to
the polarization $\lambda^m$ on $X^m$ \cite[Sect. 1.12, pp. 320--321]{ZarhinG}
such that the corresponding Riemann form
$$e_{\lambda^m}:V_{\ell}(X)^m\times V_{\ell}(X)^m\to \Q_{\ell}$$
is the direct orthogonal sum of $m$ copies of $e_{\lambda}$. Clearly,
$e_{\lambda^m}$ is $\sp(V_{\ell}(X),e_{\lambda})$-invariant, i.e.,
$$e_{\lambda^m}(ux,y)+e_{\lambda^m}(x,uy)=0 \quad \forall u\in
\sp(V_{\ell}(X),e_{\lambda}); \ x,y \in V_{\ell}(X)^m.$$ It is also clear that
every alternating $\sp(V_{\ell}(X),e_{\lambda})$-invariant bilinear form on
$V_{\ell}(X)^m$ is of the form $e_{\lambda^m}(Bx,y)$ where
$$B \in \M_m(\Q_{\ell}) \subset
\M_m(\End_{\Q_{\ell}}(V_{\ell}(X)))=\End_{\Q_{\ell}}(V_{\ell}(X)^m)$$ is a {\sl
symmetric} square matrix of size $m$ with entries from $\Q_{\ell}$.
\end{rem}

\begin{thm}
\label{Zchar0}
 Suppose that  $F$ is a field that is finitely generated over
$\Q$. Suppose that $X$ is an abelian variety of positive dimension $g$ over
$F$. Assume additionally that $X$ enjoys the following properties:

\begin{itemize}
\item $\End(X)=\Z$.

\item $\tilde{G}_{2,X,F}$ is isomorphic either to $\ST_{2g}$ or to $\A_{2g}$.

\item $g \ge 5$.
\end{itemize}

Then $\g_{\ell,X}=\Q_{\ell}\I\oplus \sp(V_{\ell}(X),e_{\lambda})$ for all
$\ell$. In particular, $G_{{\ell},X,F}$ is an open subgroup in
$\Gp(V_{\ell}(X),e_{\lambda})$.
\end{thm}

\begin{proof}
The openness property follows from the coincidence of the corresponding
$\Q_{\ell}$-Lie algebras. So, it is suffices to check that
$\g_{\ell,X}=\Q_{\ell}\I\oplus \sp(V_{\ell}(X),e_{\lambda})$.

 Replacing if necessary, $F$ by its suitable quadratic
extension, we may and will assume that $\tilde{G}_{2,X,F}=\A_{2g}$.

 By a theorem of Bogomolov \cite{Bogomolov1,Bogomolov2},
$$\g_{\ell,X}\supset \Q_{\ell}\I.$$
 By a theorem of Faltings \cite{F1,F2} for every finite algebraic extension
$F_1$ of $F$ the $\Gal(F_1)$-module $V_{\ell}$ is semisimple and
$$\End_{\Gal(F_1)}(V_{\ell}(X))=\End_{F_1}(X)\otimes\Q_{\ell}=\Z\otimes\Q_{\ell}=\Q_{\ell}.$$
In other words, the $\Gal(F_1)$-module $V_{\ell}$ is absolutely simple. Notice
that $\rho_{\ell,X}(\Gal(F_1))$ is an open subgroup of finite index in
$G_{{\ell},X,F}$. Conversely, every  open subgroup of finite index in
$G_{{\ell},X,F}$ coincides with $\rho_{\ell,X}(\Gal(F_1))$ for some finite
algebraic extension $F_1$ of $F$. It follows that the $\Q_{\ell}$-Lie algebra
$\g_{\ell,X}$ is reductive and the $\g_{\ell,X}$-module $V_{\ell}(X)$ is
absolutely simple, i.e.,
$$\End_{\g_{\ell,X}}(V_{\ell}(X))=\Q_{\ell}.$$
It follows from Lemma \ref{groupal} that $\g_{\ell,X}=\Q_{\ell}\I\oplus
\g^{ss}_{\ell}$ where $\g^{ss}_{\ell}$ is a semisimple $\Q_{\ell}$-Lie algebra
such that
$$\g^{ss}_{\ell}\subset \sp(V_{\ell}(X),e_{\lambda})$$ and the
$\g^{ss}_{\ell}$-module $V_{\ell}(X)$ is absolutely simple.
 Pick a finite algebraic field extension
$L/\Q_{\ell}$ such that the semisimple $L$-Lie algebra
$\g=\g^{ss}_{\ell}\otimes_{\Q_{\ell}} L$ splits. By a theorem of Pink
\cite{Pink}, all simple factors of $\g$ are classical Lie algebras and the
highest weight of the simple $\g$-module $V_{\ell}(X)\otimes_{\Q_{\ell}} L$ is
minuscule.

Now let us consider the case of $\ell=2$. Applying Corollary \ref{cor0}, we
conclude that $\g^{ss}_{2}$ is an absolutely simple $\Q_2$-Lie algebra and
therefore $\g$ is an absolutely simple $L$-Lie algebra. It follows from the
theorem of Pink that $\g$ is a classical simple Lie algebra and the highest
weight of the simple $\g$-module $V_{2}(X)\otimes_{\Q_{2}}L$ is  minuscule.
Applying Corollary \ref{LieSP}, we conclude that
$\g^{ss}_{2}=\sp(V_2(X),e_{\lambda})$ and $\g_{2,X}=\Q_{2}\I\oplus
\sp(V_{2}(X),e_{\lambda})$.

Now the case of arbitrary $\ell$ follows from Lemma 8.2 in \cite{ZarhinMMJ}.
\end{proof}

\begin{thm}
\label{MON} Suppose that $K$ is a field that is finitely generated over $\Q$.
Suppose that $f(x)\in K[x]$ is a polynomial of  degree $n\ge 10$ such that
 $f(x)=(x-t)h(x)$ with $t \in K$ and $h(x) \in
K[x]$. Suppose that $\Gal(h)$ is either the full symmetric group $\ST_{n-1}$ or
the alternating group $\A_{n-1}$. Let  $C_f$ be the hyperelliptic curve
$y^2=f(x)$, let $J(C_f)$ be its jacobian and $\lambda$ the principal
polarization on  $J(C_f)$ attached to the theta divisor. Then for all primes
$\ell$
$$\g_{\ell,J(C_f)}=\Q_{\ell}\I\oplus \sp(V_{\ell}(J(C_f)),e_{\lambda})$$
and the group $G_{\ell,J(C_f),K}$ is an open subgroup in
$\Gp(V_{\ell}(X),e_{\lambda})$.
\end{thm}

\begin{proof}
As above, the openness property follows from the coincidence of the
corresponding Lie algebras. So, it is suffices to check that $\g_{\ell,J(C_f)}$
coincides with $\Q_{\ell}\I\oplus \sp(V_{\ell}(J(C_f)),e_{\lambda})$.

Suppose that $n$ is {\sl even}. Then as in the Proof of Theorem \ref{even}
(Sect. \ref{intro}), the curve $C_f$ is $K$-biregularly isomorphic to
$$C_{h_2}:y_1^2=h_2(x_1)$$
where $h_2(x_1)\in K[x_1]$ is a certain degree $(n-1)$ polynomial, whose Galois
group is either $\ST_{n-1}$ or $\A_{n-1}$ respectively. Since $n-1\ge 9$, the
assertion follows from Theorem 2.4 of \cite{ZarhinMMJ} applied to the
polynomial $h_2(x_1)$.

Suppose now that $n$ is {\sl odd} and therefore $n \ge 11$.
 We have $n-1=2g$ where $g$ is an integer that is greater or equal than
$5$. Replacing if necessary, $K$ by its suitable quadratic extension, we may
and will assume that $\Gal(h)=\A_{2g}$.  By Theorem \ref{oddZ},
$\End(J(C_f))=\Z$. By Lemma \ref{order2}, $\tilde{G}_{2,J(C_f),\ell}=\Gal(h)$
and therefore $\tilde{G}_{2,J(C_f),K}=\A_{2g}$. As we have seen in Section
\ref{tate}, there is a continuous surjective homomorphism
$$\pi_{2,J(C_f),K}:G_{2,J(C_f),K}\twoheadrightarrow
\tilde{G}_{2,J(C_f),K}=\A_{2g}.$$

Now the result follows from Theorem \ref{Zchar0} applied to $F=K$ and
$X=J(C_f)$.
\end{proof}

\section{Tate classes}
\label{tateC}

The aim of this and next Sections is  to use the classical invariant theory of
symplectic groups \cite{Howe} in order to deduce from results of the previous
Section the validity of certain important conjectures for  our $J(C_f)$ and its
self-products.

\begin{thm}
 Suppose that $K$ is a field that is finitely generated over $\Q$.
Suppose that $f(x)\in K[x]$ is a polynomial of  degree $n\ge 10$ such that
 $f(x)=(x-t)h(x)$ with $t \in K$ and $h(x) \in
K[x]$. Suppose that $\Gal(h)$ is either the full symmetric group $\ST_{n-1}$ or
the alternating group $\A_{n-1}$.

Let  $C_f$ be the hyperelliptic curve $y^2=f(x)$ and $J(C_f)$  its jacobian.
Let $K^{\prime}$ be a finite algebraic extension of $K$.

Then for all primes $\ell$ and on each self-product $J(C_f)^m$ of $J(C_f)$
every $\ell$-adic Tate class over $K^{\prime}$ can be presented as a linear
combination of products of divisor classes. In particular, the Tate conjecture
holds true for all $J(C_f)^m$.
\end{thm}

\begin{proof}
 Recall \cite{Tate} that one may view Tate classes on
$J(C_f)^m$
 as tensor invariants of the Lie algebra $\g_{\ell,J(C_f)}\bigcap
 \sp(V_{\ell}(J(C_f),e_{\lambda})$ in
 $$V_{m,i}:=
\Hom_{\Q_{\ell}}(\wedge^{2i}_{\Q_{\ell}}(V_{\ell}(J(C_f))^m),\Q_{\ell}).$$
 By Theorem \ref{MON},
 $$\g_{\ell,J(C_f)}\bigcap
 \sp(V_{\ell}(J(C_f)),e_{\lambda})=$$
 $$[\Q_{\ell}\I\oplus
 \sp(V_{\ell}(J(C_f)),e_{\lambda})]\bigcap
 \sp(V_{\ell}(J(C_f)),e_{\lambda})=
 \sp(V_{\ell}(J(C_f)),e_{\lambda}).$$
The invariant theory for symplectic groups (\cite[Th. 2 on p. 543]{Howe},
\cite{Ribet2}; see also \cite{ZarhinEssen}) tells us that every
 $\sp(V_{\ell}(J(C_f)),e_{\lambda})$-invariant in $V_{m,i}$ could be
presented as a linear combination of exterior products of
$\sp(V_{\ell}(J(C_f)),e_{\lambda})$-invariants in $V_{m,1}$, i.e., of
$\sp(V_{\ell}(J(C_f)),e_{\lambda})$-invariant alternating bilinear forms on
$$V_{\ell}(J(C_f))^{m}=\bigoplus_{j=1}^{m} V_{\ell}(J(C_f)).$$
The description of those alternating invariant bilinear forms given in Remark
\ref{selfproduct} implies that they all are linear combinations of divisors
classes on $J(C_f))^{m}$ with coefficients in $\Q_{\ell}$. It follows that
  each  $\ell$-adic Tate class can be presented
 as a a linear combination of products of divisor classes and therefore is
 algebraic.
\end{proof}

\begin{rem}
In codimension $1$ the Tate conjecture for all abelian varieties over $K$ is
proven by Faltings \cite{F1,F2}.
\end{rem}

\section{Hodge classes}
\label{hodge}

Let $X$ be a complex abelian variety of positive dimension,  Let
$$V_{\Q}=V_{\Q}(X):=\H_1(X(\C),\Q)$$ be the first rational homology group of the complex
torus $X(\C)$ and let
$$e_{\lambda,\Q}:V_{\Q}(X)\times V_{\Q}(X) \to \Q$$
be the alternating nondegenerate $\Q$-bilinear (Riemann) form attached to a
polarization $\lambda$ on $X$. Let $\Sp(V_{\Q}(X),e_{\lambda,\Q})\subset
\GL(V_{\Q}(X))$ be the $\Q$-algebraic symplectic group  attached to
$e_{\lambda,\Q}$. We write $\sp((V_{\Q}(X),e_{\lambda,\Q})$ for the Lie algebra
of
 $\Sp(V_{\Q}(X),e_{\lambda,\Q})$: it is an absolutely simple absolutely irreducible
 $\Q$-Lie subalgebra of  $\End_{\Q}(V_{\Q}(X))$.
Let $\Gp(V_{\Q}(X),e_{\lambda,\Q})\subset \GL(V_{\Q}(X))$ be the (connected)
$\Q$-algebraic (sub)group of symplectic similitudes attached to
$e_{\lambda,\Q}$. We have
$$\Sp(V_{\Q}(X),e_{\lambda,\Q})\subset
\Gp(V_{\Q}(X),e_{\lambda,\Q}) \subset \GL(V_{\Q}(X)).$$ The Lie algebra of
$\Gp(V_{\Q}(X),e_{\lambda,\Q})$ coincides with $\Q\I\oplus
\sp((V_{\Q}(X),e_{\lambda,\Q})$; here $\I$ is the identity map on $V_{\Q}(X)$.

We refer to \cite{Ribet2} for the definition of the Mumford--Tate group
$\MT=\MT_X$ of $X$; it is a reductive connected $\Q$-algebraic subgroup of
$\Gp(V_{\Q}(X),e_{\lambda,\Q})$. We have
$$\MT_X\subset \Gp(V_{\Q}(X),e_{\lambda,\Q})\subset
\GL(V_{\Q}(X)).$$  We write $\mt_X$ for the Lie algebra of $\MT_X$: it is a
reductive algebraic $\Q$-Lie subalgebra of $\End_{\Q}(V_{\Q}(X))$. It is
well-known \cite{Ribet2} that
$$\Q\I\subset \mt_X\subset \Q\I\oplus\sp(V_{\Q}(X),e_{\lambda,\Q})\subset\End_{\Q}(V_{\Q}(X)).$$
We refer to \cite[Sect. 3 and 4]{SerreKyoto} for the precise statement and a
discussion of the Mumford--Tate conjecture for abelian varieties. (See also
\cite{ZarhinW}.)

\begin{thm}
Suppose that $f(x)\in \C[x]$ is a polynomial of degree $n\ge 10$ without
multiple roots. Let $C_f$ be the hyperelliptic curve $y^2=f(x)$ and $X=J(C_f)$
its jacobian, viewed as a complex abelian variety provided with the canonical
principal polarization $\lambda$ attached to the theta divisor. Let
$$e_{\lambda,\Q}:V_{\Q}(X)\times V_{\Q}(X) \to \Q$$
be the alternating nondegenerate $\Q$-bilinear (Riemann) form attached to
$\lambda$.

Suppose that all the coefficients of $f(x)$ lie  in a subfield $K\subset \C$
and $f(x)=(x-t)h(x)$ with $t\in K, \ h(x)\in K[x]$. Assume also that $K$ is
finitely generated over $\Q$ and the Galois group of $h(x)$ over $K$ is either
$\ST_{n-1}$ or $\A_{n-1}$.

Then:

\begin{itemize}
\item The Mumford--Tate group of $X$ coincides with
$\Gp(V_{\Q}(X),e_{\lambda,\Q})$.

\item Each Hodge class on every self-product $X^m$ of $X$ can be presented as a
linear combination of products of divisor classes. In particular, the Hodge
conjecture holds true for all $X^m$.

\item
 The Mumford--Tate conjecture holds true for $J(C_f)$. (Here $J(C_f)$ is viewed
 as an abelian variety over $K$.)
\end{itemize}

\end{thm}

\begin{proof}
We use the arguments from \cite[p. 429]{ZarhinMMJ}. Let $\bar{K}\subset \C$ be
the algebraic closure of $K$ in $\C$. For each prime $\ell$ let us consider the
$\Q_{\ell}$-vector space
$$\Pi_{\ell}=V_{\Q}(X)\otimes_{\Q}\Q_{\ell}.$$

Then there is a well-known (comparison) isomorphism of $\Q_{\ell}$-vector
spaces \cite{MumfordAV,ZarhinDuke}
$$\gamma_{\ell}:\Pi_{\ell} \cong V_{\ell}(J(C_f))$$
such that, by a theorem of Piatetski-Shapiro--Deligne--Borovoi
\cite{Deligne,SerreKyoto},
$$\gamma_{\ell}\g_{\ell,X}\gamma_{\ell}^{-1}\subset
\mt_X\otimes_{\Q}\Q_{\ell}\subset
[\Q\I\oplus\sp((V_{\Q}(X),e_{\lambda,\Q})]\otimes_{\Q}\Q_{\ell}.$$ It follows
from Theorem \ref{MON}, the $\Q_{\ell}$-dimension of $\g_{\ell,X}$ and the
$\Q$-dimension of $\Q\I\oplus\sp(V_{\Q}(X),e_{\lambda,\Q})$ do coincide. It
follows that
$$\gamma_{\ell}\g_{\ell,X}\gamma_{\ell}^{-1}=
\mt_X\otimes_{\Q}\Q_{\ell}=
[\Q\I\oplus\sp(V_{\Q}(X),e_{\lambda,\Q})]\otimes_{\Q}\Q_{\ell}.$$ The first
equality means that the Mumford-Tate conjecture holds true for $J(C_f)$. The
second equality means that the $\Q$-dimensions of $\mt_X$ and
$\Q\I\oplus\sp(V_{\Q}(X),e_{\lambda,\Q})$ do coincide. Since $\mt_X\subset
\Q\I\oplus\sp(V_{\Q}(X),e_{\lambda,\Q})$, we conclude that
$$\mt_X=
\Q\I\oplus\sp(V_{\Q}(X),e_{\lambda,\Q}).$$ This  implies that
$\MT_X=\Gp(V_{\Q}(X),e_{\Q})$, because their $\Q$-Lie algebras do coincide.

Recall \cite{Ribet2} that the Hodge classes on a self-product $X^m$ of $X$ can
be viewed as tensor invariants of the $\Q$-Lie algebra
$$\mt_X\bigcap \sp(V_{\Q}(X),e_{\lambda,\Q})=[\Q\I\oplus\sp(V_{\Q}(X),e_{\lambda,\Q})]\bigcap
\sp(V_{\Q}(X),e_{\lambda,\Q})=$$ $$\sp(V_{\Q}(X),e_{\lambda,\Q})$$ in the
$\Q$-vector space
$$V_{\Q,m,i}:=\Hom_{\Q}(\wedge^{2j}_{\Q}(V_{\Q}(X)^m),\Q)$$ where
$$V_{\Q}(X)^m:=\bigoplus_{j=1}^{m}V_{\Q}(X).$$
The invariant theory for symplectic groups (\cite[Th. 2 on p. 543]{Howe})
implies that every
 $\sp(V_{\Q}(J(C_f)),e_{\lambda,\Q})$-invariant  tensor in $V_{\Q,m,i}$ could be
presented as a linear combination of exterior products of
$\sp(V_{\Q}(J(C_f)),e_{\lambda,\Q})$-invariants in $V_{\Q,m,1}$, i.e., of
$\sp(V_{\ell}(J(C_f)),e_{\lambda,\Q})$-invariant alternating bilinear forms on
$$V_{\Q}(J(C_f))^{m}=\bigoplus_{j=1}^{m} V_{\Q}(J(C_f)).$$
This implies \cite{Ribet2} that each Hodge class on  $X^m$  can be presented as
a linear combination of products of divisor classes. In particular, every Hodge
class on $X^m$ is algebraic, i.e., the Hodge conjecture is true for $X^m$ in
all dimensions.
\end{proof}

\section{Appendix}

\centerline{\bf Corrigendum to \cite{ZarhinMMJ}}

Page 408, Proposition 3.3(i), the first sentence. In addition,
$\pi^{\prime}(H)$ is normal in $\pi^{\prime}(G)$ and $H^{\prime}$ is normal in
$G^{\prime}$.

Page 409, line 3. The $R_j$  should be $S_j$.

Page 409, Corollary 3.4(ii). In addition, $H_{-1,\alpha}$ is normal is
$G_{-1,\alpha}$.

Page 421, Step 4. The $\Ad_{i,E}$ is the canonical central isogeny $\r_{i} \to
\r_{i}^{\Ad}$ of absolutely simple $E$-algebraic groups.


\centerline{\bf Corrigendum to \cite{ZarhinMA}}

{Theorem 1.1 on page 408}. The $B(h)$ in assertion (i) should be $B(f)$.

 {Lemma 3.5 on page 419}. The $\End^0(Y)$ should be $\End^0(X)$ and the
$Y$ should be $X$.

{Theorem 3.11 on page 422}. The $Y$ should satisfy either condition (i) or
condition (ii) (not necessarily both).


{Proof of Theorem 3.11 on page 423} should be modified as follows. The
assertion (in the last sentence of the first paragraph)  that $\Aut(X)$ is a
finite cyclic group does {\sl not} follow from Theorem 3.7. However, every
finite subgroup of $\Aut(X)$ is cyclic, because $\Aut(X)\subset \End^0(X)^{*}$
and, by Theorem 3.7, $\End^0(X)$ is a field. On the other hand, since the
isomorphism $u$ is always defined over a finite Galois extension of the ground
field $F$, the image of the cocycle-homomorphism $c$ is a finite subgroup of
$\Aut(X)$ and therefore is a (finite) cyclic group. The rest of the proof
remains unchanged.

Page 423: in the (third) displayed formula in the middle of page (the
subscript) $Y$ should be $X$, as well as in the formula in the previous line.

{Proof of Theorem 1.1 on page 429}. The $B(h)$ in the first displayed formula
should be $B(f)$. The $G_{B(h)}$ in the next line should be $G_{B(f)}$.

\end{document}